\documentclass[pdflatex,sn-mathphys-num]{sn-jnl}

\usepackage{graphicx}%
\usepackage{multirow}%
\usepackage{amsmath,amssymb,amsfonts}%
\usepackage{amsthm}%
\usepackage{mathrsfs}%
\usepackage[title]{appendix}%
\usepackage{xcolor}%
\usepackage{extarrows}
\usepackage{amsbsy}
\usepackage{mathtools}
\usepackage{textcomp}%
\usepackage{manyfoot}%
\usepackage{booktabs}%
\usepackage{algorithm}%
\usepackage{algorithmicx}%
\usepackage{algpseudocode}%
\usepackage{listings}%

\theoremstyle{thmstyleone}%
\newtheorem{theorem}{Theorem}
\newtheorem{lemma}{Lemma}
\newtheorem{corollary}{Corollary}
\newtheorem{proposition}{Proposition}

\theoremstyle{thmstyletwo}%
\newtheorem{remark}{Remark}%

\theoremstyle{thmstylethree}%
\newtheorem{definition}{Definition}%

\begin{document}

\title[Article Title]{Error bounds of Median-of-means estimators with VC-dimension}


\author[1]{\fnm{Yuxuan} \sur{Wang}}\email{wangyuxuan@zju.edu.cn}

\author[2]{\fnm{Yiming} \sur{Chen}}\email{chenyiming960212@mail.sdu.edu.cn}

\author[2]{\fnm{Hanchao} \sur{Wang}}\email{wanghanchao@sdu.edu.cn}

\author*[1,3]{\fnm{Lixin} \sur{Zhang}}\email{stazlx@zju.edu.cn}

\affil*[1]{\orgdiv{School of Mathematical Sciences}, \orgname{Zhejiang University}, \orgaddress{\city{Hangzhou}, \postcode{310027}, \country{China}}}

\affil[2]{\orgdiv{Institute for Financial Studies}, \orgname{Shandong University}, \orgaddress{\city{Jinan}, \postcode{250100}, \country{China}}}

\affil[3]{\orgdiv{School of Statistics and Mathematics}, \orgname{Zhejiang Gongshang University}, \orgaddress{ \city{Hangzhou}, \postcode{310018}, \country{China}}}


\abstract{We obtain the upper error bounds of robust estimators for mean vector, using the median-of-means (MOM) method. The method is designed to handle data with heavy tails and contamination, with only a finite second moment, which is weaker than many others, relying on the VC dimension rather than the Rademacher complexity to measure statistical complexity. This allows us to implement MOM in covariance estimation, without imposing conditions such as \(L\)-sub-Gaussian or \(L_{4}-L_{2}\) norm equivalence. In particular, we derive a new robust estimator, the MOM version of the halfspace depth, along with error bounds for mean estimation in any norm.}




\maketitle

\section{Introduction}

Inspired by applications in machine learning and data science, there has been a growing interest in constructing $\mu$ mean estimators in recent years. As the most basic method of estimation, the sample mean $\bar{\mu}_N=\frac{1}{N} \sum_{i=1}^n X_i$ on a sample $\left(X_1, \ldots, X_N\right)$ of $N$ independent and identically distributed random variables possesses favorable statistical properties established by the central limit theorem. However, the asymptotic properties often require a large sample size in practical applications, significantly increasing the difficulty level. Alternatively, non-asymptotic estimators with faster convergence rates have emerged, implying the need for fewer samples.

Simultaneously, in situations where the distribution exhibits heavy-tailed characteristics or outliers present in the data, the empirical mean may no longer be sufficient to meet the requirements. There is an urgent need to enhance the quality of mean estimation, particularly in a non-asymptotic context. It is worth mentioning that in mean estimation, statistical optimality, including convergence rates and computational complexity, are crucial factors. We will primarily focus on estimators that maintain high accuracy while providing substantial confidence in mean estimation.
Unlike asymptotic estimators, we consider a non-asymptotic estimator known as the "$L$-sub-Gaussian" estimator as below. For a mean estimator $\widehat{\mu}_N$ and any $\delta\in(0,1)$, let $\sigma^2$ be the variance, then there exists a constant \(L > 0\), for any sufficiently large sample size \(N\), the following inequality holds with at least \(1 - \delta\) probability:
$$
\left|\widehat{\mu}_N-\mu\right| \leq \frac{L \sigma \sqrt{\log (2 / \delta)}}{\sqrt{N}}.
$$
This definition comes from the accuracy of empirical mean estimation in the context of sub-Gaussian distributions. Many well-known robust mean estimators exhibit this property, with the most common being the median-of-means (MOM) estimator. MOM method has experienced rapid development, as evidenced by works such as \cite{LO,devroye2016sub,LM2}.

The estimation of vector means and real-valued sample means are fundamentally different, as illustrated in \cite{LM1}. The former involves transforming into a problem of concentration inequalities for the upper bounds of a stochastic process indexed by vectors in $\mathbb{R}^d$. This represents an essential distinction and also marks a departure from previous work.

For the estimation of the mean in the multivariate sub-Gaussian case, the empirical mean of independent and identically distributed samples with mean \( \mu \) and covariance matrix \( \Sigma \) satisfies, with at least \( 1-\delta \) probability:

\begin{equation}\label{sg}
\left\|\dfrac{\sum_{i=1}^{N}X_i}{N}-\mu\right\| \leq\sqrt{\dfrac{\text{Tr}(\Sigma)}{N}}+ \sqrt{\dfrac{\left\| \Sigma\right\| \log(1/\delta)}{N}},
\end{equation}
where $\operatorname{Tr}(\Sigma)$ is the trace of $\Sigma$ and $\|\Sigma\|$ is the operator norm of $\Sigma$. A mean estimator is considered sub-Gaussian, as defined in \cite{LM}, if it satisfies an inequality of the form above (with possibly different constant factors). Lugosi and Mendelson \cite{lugosi2019near} used the MOM estimator and obtained near-optimal confidence bounds for mean estimation under general heavy-tailed conditions, specifically when only the second moment is finite. This means that for \( \delta \in (0,1) \), the optimal confidence upper bound holds with at least a probability of \( 1-\delta \):$$\left\|\widehat{\mu}_N-\mu\right\| \leq \frac{c}{\sqrt{N}}\left(\max \left\{\mathbb{E}\left\|Y_N\right\|, \mathbb{E}\|G\|+R \sqrt{\log (2 / \delta)}\right\}\right),
$$
where $c$ is an absolute constant and
$$
R=\sup _{x^* \in \mathcal{B}^{\circ}}\left(\mathbb{E}\left(x^*(X-\mu)\right)^2\right)^{1 / 2},
$$
$$
Y_N=\frac{1}{\sqrt{N}} \sum_{i=1}^N \varepsilon_i\left(X_i-\mu\right),
$$
where $\mathcal{B}^{\circ}$ is the unit ball of the dual space to $\left(\mathbb{R}^d,\|\|\right)$, $\left(\varepsilon_i\right)_{i=1}^N$ are i.i.d. symmetric $\{-1,1\}$-valued random variables that are also independent of $\left(X_i\right)_{i=1}^N$, and $\mathbb{E}\left\|Y_N\right\|/\sqrt{N}$ is called the Rademacher complexity. Also, observe that by the central limit theorem, $Y_N$ tends, in distribution, to the centered Gaussian random vector $G$ that has the same covariance as $X$.

Mendelson and Zhivotovskiy \cite{10.1214/19-AOS1862} furthered this line of research by constructing estimators for the covariance matrix of random vectors that are robust to heavy tails and outliers. However, their approach imposes stricter conditions, such as the \( L_{4}-L_{2} \) norm equivalence.
Some other covariance estimators adopt a two-step approach \cite{ke2019user}: firstly, estimate \( \mu \) using MOM or other methods, and then employ truncation techniques to estimate the covariance, with the aim of mitigate the influence of heavy tails. This two-step method is often used to improve the robustness and precision of covariance matrix estimation in the presence of heavy-tailed distributions or outliers.

\begin{remark}[Catoni estimator]
    As one of the first studies on sub-Gaussian mean estimators, Catoni \cite{catoni2012challenging} introduced a sharp example for distributions with known variances and distributions with finite fourth moments and known upper bounds of kurtosis. Using a specific function \( \phi \), the M-estimator, also known as the Catoni estimator, provides a well-performed confidence interval, naturally extended to mean estimation of heavy-tailed random vectors \cite{catoni2017dimension} and covariance matrix estimation \cite{wei2017estimation,minsker2018sub}. However, the construction of Catoni estimator inevitably requires the distribution variance \( \sigma^2 \) or \( \Sigma \) to be known.  
\end{remark} 
In computer science and logic design, Boolean functions are basic, representing functions with binary output. In machine learning, they are used to model simple classification problems. The VC dimension, first introduced by Vapnik and Chervonenkis \cite{vapnik2013nature}, measures the maximum complexity that a model's hypothesis space can handle, especially in classification tasks. While it is primarily used for assessing classifier complexity, directly computing the VC dimension can be challenging, especially in high-dimensional and complex models. It serves as a theoretical guide for understanding the learning and generalization of a model. 

Depersin \cite{depersin2020robust} introduced a novel general approach to constrain the estimation error of MOM estimators. The author applied VC dimension instead of Rademacher complexity to measure statistical complexity, which does not take into account
the unknown structure of the covariance matrix, but is related only to the dimension
of the dual space. 

In the context of multivariate analysis, the varying definitions of the median give rise to distinct MOM estimators. Among those, geometric median is proved, by Minsker \cite{minsker2015geometric}, to be available in constructing robust MOM estimators.
Depersin and Lecu{\'e} \cite{depersin2023robustness} then discussed the construction of sub-Gaussian estimators of a mean vector by VC dimension. 

Furthermore, ever since Tukey \cite{tukey1975mathematics} introduced the concept of data depth (also known as halfspace depth), it has emerged as a fundamental tool to assess the centrality of data points in multivariate datasets. Consequently, in this paper, we endeavor to delve into the error bounds achievable by a novel estimator: the MOM adaptation of Tukey's median. We anticipate exploring the bounds within certain confidence levels, leveraging the VC dimension as a guiding framework. Through this exploration, our objective is to shed light on the robustness and efficacy of MOM estimators in practical data analysis scenarios.

\paragraph{Motivation.}
In many contemporary data science and machine learning pipelines one is confronted with high--dimensional observations, \emph{heavy--tailed} distributions, possibly \emph{contamination} and only a \emph{moderate} number of samples.
A single anomalous data point may distort the empirical mean or covariance by orders of magnitude, yet these basic quantities constitute indispensable building blocks for downstream procedures such as principal component analysis, classification and linear regression.
We therefore require estimators that
\begin{enumerate}
	\item match the statistical efficiency of the empirical mean in well--behaved (sub--Gaussian) regimes,
	\item remain reliable under heavy--tailed distributions or an unknown fraction of outliers, 
	\item come with explicit, finite sample (non--asymptotic) deviation guarantees.
\end{enumerate}
These considerations motivate our investigation of MOM based procedures and the development of a MOM adaptation of Tukey's halfspace median that attains sub--Gaussian deviation bounds without imposing restrictive moment or tail assumptions.

\bigskip
\paragraph{Our contributions.}
Building on the recent advances of Lugosi, Mendelson, Depersin, and others, our goal is to develop a unified MOM framework that (a) requires only finite second moments, (b) is robust to contamination, and (c) extends seamlessly from mean to covariance estimation and further to PCA.
The main contributions of this paper are:
\begin{itemize}
	\item We construct a MOM estimator of the mean vector that achieves nearly sub--Gaussian deviation bounds under merely second--moment assumptions.
	\item We show that the same methodology provides a theoretical foundation for robust estimation of the covariance matrix without relying on $L$--sub--Gaussian or $L_{4}$--$L_{2}$ norm--equivalence conditions.
	\item We introduce a MOM version of Tukey's median and establish its non--asymptotic error bounds via VC--dimension arguments, suggesting its potential usefulness for robust PCA.
\end{itemize}

The structure of this paper is as follows. In Section \ref{2}, we will provide the necessary symbol explanations and introduce the definitions and lemmas; in Section \ref{3} and Section \ref{4}, we present the error bounds of the mean estimation introduced by \cite{lugosi2019near} and covariance estimation, respectively; finally, in Section \ref{5}, we will give a Tukey MOM estimator.

\section{Preliminary}\label{2}
\subsection{Notation}
In this paper, we assume that the covariance matrix of interest is non-degenerate. We use $||\cdot||$ to represent a norm on $\mathbb{R}^d$, and assume the existence of an inner product $\left\langle \cdot, \cdot \right\rangle$ that induces this norm. $||\cdot||_*$ denotes its dual norm. $\mathcal{B}$ represents the unit ball of the norm $||\cdot||$, and $\mathcal{B}^*$ represents the unit ball of the norm $||\cdot||_*$. $\mathcal{B}_0^*$ is defined as the set of extreme vectors of $\mathcal{B}^*$. We also introduce the operator norm $||A||=\sup_{u \in \mathcal{B}^*} ||Au||_2$, where $||\cdot||_2$ is the Euclidean norm on $\mathbb{R}^d$. In particular, for a vector $u = (u_i)$, $||u||_2 = \sqrt{\sum_i u_i^2}$ represents the $\ell^2$ norm. The set $S^{d-1} = \{u \in \mathbb{R}^{d} : ||u|| = 1\}$ is the unit sphere in $\mathbb{R}^{d}$. For a matrix $A = (A_{ij})$, when $A = A^T \in \mathbb{R}^{p\times p}$ is symmetric, we use $\lambda_j(A)$ to denote its $j$th largest singular value. The operator norm of $A$ is represented as $||A||_{op} = \lambda_{1}(A)$, and the Frobenius norm is denoted as $||A||_{F} = \sqrt{\sum_{ij} A^2_{ij}}$. 

Given an integer $d$ and $ a, b \in \mathbb{R} $, we use $[d]$ to denote the set $\{1, 2, \dots, d\}$ and write $a \vee b = \max(a, b)$ and $ a\wedge b = \min(a, b) $. For two non-negative sequences $ \{a_{n}\}, \{b_{n}\} $, for some constant $ C > 0 $ independent of $ n $, $ a_n\lesssim b_n $ means $ a_n\leq C b_n $, and $ a_n\gtrsim b_n $ means $  a_n\geq C b_n $.  Throughout the paper, $ C, c $ and their variations, whose specific values may vary, represent universal constants independent of $ n $. Additionally, the indicator function $ \mathbf{1}_B(\cdot) $ is defined as$$
\mathbf{1}_B(u)= \begin{cases}1, & \text { if } u\in B, \\ 0, & \text { if } u \notin B .\end{cases}
$$

\subsection{VC dimension}\label{vc}
Boolean classes $ \mathcal{F} $ arise in the problem of classification, where $ \mathcal{F} $ can be taken to consist of all functions
$ f $ of the form $ \mathbf{1}_{\{g(X)\ne Y\} }$ for mappings $g$. VC dimension was first studied by Vapnik and Cervonenkis
in the 1970s, and let us recall the classical definitions.
\begin{definition}[Boolean function]
	Every $ f \in \mathcal{F} $, taking values in $ \{0, 1\} $, is called a boolean function. And $ \mathcal{F} $ is called a boolean class of functions.
\end{definition}

\begin{definition}[VC dimension]
	Let $  \mathcal{C} $ be a class of subsets of any set $ \mathcal{X} $. We say that $  \mathcal{C} $ picks out a certain subset from $ \{x_1, . . . , x_n\} $ if this can be formed as a set of the form $  C \cap  \{x_1, . . . , x_n\} $ for some $ C\in  \mathcal{C} $. The collection $ \mathcal{C} $ is said to shatter $ \{x_1, . . . , x_n\} $ if each of its $ 2^n $ subsets can be picked out by $  \mathcal{C} $. The VC dimension  $\operatorname{VC}(\mathcal{C})$ is the largest cardinality of a set shattered by $  \mathcal{C} $, more formally, \[\operatorname{VC} (\mathcal{C}) = \sup\left\lbrace n : \max_{x_1 ,...,x_n\in\mathcal{X}}\#\{C\cap \{x_1, . . . , x_n\}:C\in \mathcal{C}\}= 2^n\right\rbrace ,\] and in particular, $\operatorname{VC}(\mathcal{C})=-1$ if $\mathcal{C}$ is empty.
\end{definition}
The definition of VC dimension can be easily extended to a function class $ \mathcal{F} $ in which every
function $ f $ is binary-valued, taking the values within $ \{0, 1\} $.  In this case, we define \[\operatorname{VC} (\mathcal{F}) = \sup\left\lbrace n : \max_{x_1, . . . , x_n\in\mathcal{X}}\#\{\left( f(x_1) ,...,f(x_n)\right) :f\in \mathcal{F}\}= 2^n\right\rbrace.\]
In particular, we derive the equivalent definition of the VC dimension as for a set $ C $ be a subset of Euclidean space $ E $, the VC dimension of the set of half-spaces generated by the vectors of $ C $,\[\operatorname{VC}(C)=\operatorname{VC}\left( \left\lbrace x\in E\rightarrow\mathbf{1}_{\left\langle x,v\right\rangle \geq 0}: v\in C\right\rbrace\right) . \]

\subsection{Contamination model}
As in practice, sometimes we cannot directly observe the vectors $Y_1, \ldots, Y_N$. Instead, this dataset may already be contaminated or corrupted. One of the most famous examples is the so-called Huber's contamination model. In this setting, instead of observing samples directly from the true distribution $P$, we observe samples drawn from $P_\varepsilon$, which for an arbitrary distribution $ Q $ is deﬁned as a mixture model,\[P_\varepsilon = (1-\varepsilon)P +\varepsilon Q.\]   This setting is called the $ \varepsilon $-contamination model, first proposed in a groundbreaking paper by Huber \cite{huber1964robust}.

More generally, our problem is that the contamination may be adversarial \cite{dalalyan2022all}. This means that when an $\varepsilon$ fraction of all observed values is maliciously tampered with by an adversary, who is aware of both the "clean" samples and our estimators, there exists a (possibly random) set $\mathcal{O}$ such that for any $i \in \mathcal{O}^c, X_i=Y_i$. Here, the size of $\mathcal{O}$ satisfies $|\mathcal{O}| \leq \lfloor \varepsilon N \rfloor$. Thus, the dataset we observe is $\left\{X_i: i=1, \ldots, N\right\}$, and this model is commonly referred to as a strong contamination model. The contaminated samples $\left\{X_i: i=1, \ldots, N\right\}$ will be called $\varepsilon$-contaminated samples. Furthermore, our task is to recover $\mu$ and $\Sigma$.

\subsection{Median of mean}
Recall the definition of the classic median-of-means (MOM). First, we randomly divide the data into $K$ equally sized blocks $B_1, \ldots, B_K$ (if $K$ does not divide $N$ evenly, we discard some data). Then we calculate the empirical mean within each block. For $k=1, \ldots, K$,
$$
\bar{X}_k=\frac{1}{m} \sum_{i \in B_k} X_i.
$$

	\begin{definition}[Univariate median]
		In the one-dimensional case, for $x_{1},\dots,x_{n} \in \mathbb{R}$, $\text{Med}(x_{k})=x_{i}$, such that
		\[ \#\left\{j \in[n]: x_j \leq x_i\right\} \geq \frac{n}{2} \text { and }\#\left\{j \in[n]: x_j \geq x_i\right\} \geq \frac{n}{2},\]
		where $ \#(\cdot) $ denotes the cardinality of the set, and if there are multiple $i$ satisfying this condition, the median is defined as the smallest among them.
	\end{definition}

Let the MOM estimator be $\tilde \mu_{0}\coloneqq \text{Med}(\bar{X}_k)$, then it can be shown that, under suitable second-moment conditions, $\tilde \mu_{0}$ is a sub-Gaussian estimator.

\section{Mean estimation }\label{3}

Now, let $Y_1, \ldots, Y_N$ denote $N$ independent and identically distributed random vectors in $\mathbb{R}^d$. Our goal is to estimate $\mathbb{E}Y_1 = \mu\in \mathcal{U}$, where $\mathcal{U}$ is a subset of $\mathbb{R}^d$, assuming that $Y_1$ has finite second moments. Define $\Sigma = \mathbb{E}\left(\left(Y_1-\mu\right)\left(Y_1-\mu\right)^T\right)$, sometimes also denoted as $\mathbb{E}\left(\left(Y_1-\mu\right) \otimes \left(Y_1-\mu\right)\right)$, to represent the unknown covariance matrix of $Y_1$. The $\varepsilon$-contaminated samples $\left\{X_i: i=1, \ldots, N\right\}$ is observed.

For the mean estimation of a multi-dimensional random vector, we have the following class of median-of-means (MOM) estimators:
\begin{definition}[MOM proposed by \cite{lugosi2019near}]
For $\epsilon>0$, the sample $\left\{X_i\right\}_{i=1}^N$ can be divided into $K$ blocks $B_k$, each of size $m=N / K$. Let $\bar{X}_k=\frac{1}{m} \sum_{i \in B_k} X_i$.
For each $x^* \in \mathcal{B}^*(\mathcal{U})$, we obtain the set
\begin{equation}\label{s}
S_{x^*}=\left\{y \in \mathcal{U}: \left| \text{Med}\left(x^*\left(\bar{X}_k\right) : k\in [K]\right)-x^*(y) \right| \leq \epsilon\right\} .   
\end{equation}
Let $\mathbb{S}(\epsilon)=\bigcap_{x^* \in \mathcal{B}_0^*(\mathcal{U})} S_{x^*}$, and $\widehat{\mu}_K(\epsilon, \delta)$ be taken as any point in $\mathbb{S}(\epsilon)$.
\end{definition}
 Intuitively, \( \mathbb{S}(\epsilon) \) represents the set of points that are ‘consistent’ with the majority of block-wise medians. For each dual direction $x^*\in \mathcal{B}^*_0(\mathcal{U})$, the constraint
 \[
 \big|\,\mathrm{Med}(x^*(\bar X_k)) - x^*(y)\,\big|\le\epsilon
 \]
 defines a closed slab around the hyperplane $x^*(y)=\mathrm{Med}(x^*(\bar X_k))$. 
 The set $\mathbb{S}(\epsilon)$ is the intersection of all such slabs, and hence forms a closed 
 (and convex, when $\mathcal{U}$ is convex) region in $\mathcal{U}$. This construction ensures that any 
 candidate $y\in \mathbb{S}(\epsilon)$ is simultaneously consistent with the robust 
 direction-wise median constraints, and under the event
 \[
 \sup_{x^*\in \mathcal{B}^*_0(\mathcal{U})}\big|\mathrm{Med}(x^*(\bar X_k))-x^*(\mu)\big|\le \epsilon,
 \]
 the true mean $\mu$ is guaranteed to belong to $\mathbb{S}(\epsilon)$ and satisfies
 $\|y-\mu\|\le 2\epsilon$ for all $y\in \mathbb{S}(\epsilon)$. Thus, $\mathbb{S}(\epsilon)$ 
 can be interpreted as a data-driven confidence region whose geometry is determined 
 by robust directional constraints and whose diameter is controlled by the tolerance $\epsilon$.
 
 \begin{remark}
 	Note that the construction of our estimator $ \widehat{\mu}_{K} $ depends on the number of splitting blocks $ K $, which is relative to the level $ \delta $.   In addition, the drawback is that this estimator is more theoretical than practical, since it is not convenient to construct a set such as $\mathbb{S}(\epsilon)$. As shown in \cite{lugosi2019near}, for every $\epsilon>0$, the sets $\mathbb{S}(\epsilon)$ are compact, nested, and nonempty for a sufficiently large $\epsilon$. Therefore, the set
 	$$
 	\mathbb{S}=\bigcap_{\epsilon>0: \mathbb{S}(\epsilon)\neq \emptyset} \mathbb{S}(\epsilon)
 	$$
 	is not empty. We can define the mean estimator as any element in $\mathbb{S}$.
 \end{remark}
It can be shown that the proposed estimator satisfies the following.
\begin{theorem}\label{m}
	For any $\delta \in\left[e^{-c N}, 1 / 2\right]$, there exists an estimator $\widehat{\mu}_\delta$ such that, with probability at least $ 1-\delta, $	$$
	\left\|\widehat{\mu}_\delta-\mu\right\| \lesssim R\left(\sqrt{\frac{\mathrm{VC}\left(\mathcal{B}_0^*(\mathcal{U})\right)}{N}}+\sqrt{\frac{\log (1 / \delta)}{N}}+\sqrt{\varepsilon}\right) ,
	$$
	where $ R^{2}=\sup _{v \in \mathcal{B}_0^*(\mathcal{U})} \mathbb{E}\left(\left\langle Y_1-\mu, v\right\rangle^2\right). $
\end{theorem}

We will show that, in fact, we obtain the concentration inequality on the MOM estimator \[\mathbb{P}\left(\left\|\mu-\widehat{\mu}\right\| \geq 8R\sqrt{\frac{K}{N}}\right) \leq \exp{(-K/128)},\]whenever $ K\geq C(\mathrm{VC}(\mathcal{B}_0^*(\mathcal{U}))\vee|\mathcal{O}|), $ where $C$ is a universal constant. Hence, by taking $K\geq 128\log (1/\delta)$, we get the upper bound of the error.

Since the result is true for any norm equipped with inner product, if we consider the Euclidean norm, it shows that with probability no more than $ 1-\delta, $ there exists an estimator such that
$$
\left\|\mu-\widehat{\mu}_\delta\right\|_{2} \lesssim\left\|\Sigma^{1 / 2}\right\|\left(\sqrt{\frac{\mathrm{VC}\left(\mathcal{B}_0^*\right)}{N}}+\sqrt{\frac{\log (1 / \delta)}{N}}+\sqrt{\varepsilon}\right) .
$$
\begin{remark} We refer to the first two terms on the right-hand side of the above inequality as the weak and strong terms, respectively: the strong term  is a global component, and the weak term with directional information, corresponds to the largest variance of a one dimensional marginal of $Y_1$, that is,  $\sup _{v \in {B}_0^*} \sigma(v):=\sup _{v \in \mathcal{B}_0^*} \sqrt{\mathbb{E}\left(\left\langle Y_1-\mu, v\right\rangle^2\right)}$ Strong-weak inequalities are an important notion in high-dimensional probability, and for more improvements, one can see \cite{lugosi2020multivariate}, which constructs an estimator that, up to the optimal strong term, performs robustly in every direction. As for the third term, the corruption error guarantee, is known to be information-theoretically
optimal, even in the infinite sample regime, according to \cite{cheng2019high}.
\end{remark}
Consider the most common scenario, where $\mathcal{U}=\mathbb{R}^d$, hence ${B}_0^*=S^{d-1}$, we can derive an upper bound with the first term of the order of $ \sqrt{\lambda_1(\Sigma)d}/\sqrt{N} $. Compared with the error of the sub-Gaussian empirical mean in (\ref{sg}), where the corresponding term is $ \sqrt{\operatorname{Tr}(\Sigma)}/\sqrt{N} $, it matches when $\Sigma \simeq \lambda \mathrm{I_d}$. 

To illustrate the utility of the VC dimension, we now provide an example in the sparse setting. Let $ \mathcal{U}_{s}=\left\lbrace y\in \mathbb{R}^{d}: \left\| y\right\|_{0}\leq s \right\rbrace  $, for $ s<d $. Take $ \mathcal{U}=\mathcal{U}_{s} $ and $ \mathcal{B}_0^*\left( \mathcal{U}\right) =S^{d-1}\cap \mathcal{U}_{2s} $ in the proof of Theorem \ref{m}. Since $ \mathrm{VC}\left(\mathcal{B}_0^*\left( \mathcal{U}\right)\right)\leq \mathrm{VC}\left( \mathcal{U}_{2s} \right)$, which can be considered as the VC dimension of the union of $ \binom{d}{2s} $ s-dimension subspaces of $ \mathbb{R}^{d} $, by Lemma \ref{v}, we have $ \mathrm{VC}\left( \mathcal{U}_{2s} \right)\leq cs\log(sd) $. As a result, we obtain the following corollary, which matches the optimal rate in a sparse setting.

\begin{corollary}
	Suppose $ \mu\in \mathbb{R}^{d},\left\| \mu\right\|_{0}\leq s  $, and $ \left\|\Sigma\right\|<\infty $, then for any $\delta \in\left[e^{-c N}, 1 / 2\right]$, there exists an estimator $\widehat{\mu}_\delta$ such that, with probability at least $ 1-\delta, $	$$
	\left\|\widehat{\mu}_\delta-\mu\right\|_{2} \lesssim \left\|\Sigma^{1 / 2}\right\|\left(\sqrt{\frac{s\log (sd)}{N}}+\sqrt{\frac{\log (1 / \delta)}{N}}+\sqrt{\varepsilon}\right) .
	$$
\end{corollary}

\section{Covariance estimation}\label{4}

Let $0<\delta<1$ and consider the given sample $X_1, \ldots, X_N$. Again, the sample $\left(X_i\right)_{i=1}^N$ can be partitioned into $K$ blocks $B_k$, each of size $m=N / K$.
Set $M_k=\frac{1}{m} \sum_{i \in B_k} {X}_i \otimes {X}_i$. Recall the well-known fact that the dual norm to the operator norm is the nuclear norm. And, since a linear functional $z$ acts on the matrix $x$ via trace duality, that is, $z(x)=[z, x]:=\operatorname{Tr}\left(z^T x\right)$. It follows that
$T=\left\{u\otimes u\mid u \in \mathcal{B}_0^*\left(\mathbb{R}^d\right)\right\}$ is the set of extreme points of the corresponding dual unit ball $B^{\circ}$. For $\epsilon>0$ and a fixed $u\in\mathcal{B}_0^*$, let $ U=u\otimes u, $ and 
$$
S_{u}(\epsilon)=\left\{Y \in \mathbb{R}^{d \times d}:\left|\left[ M_k-Y, U\right] \right| \leq \varepsilon \text { for more than } K / 2 \text { blocks }\right\} .
$$
Set
$$
S(\epsilon)=\bigcap_{U \in T} S_u(\epsilon) .
$$
The estimator $\widehat{\Sigma}_{\delta}$ is taken as any points in $ S(\epsilon) $. Again, we derive the bound of robust covariance estimator.
\begin{theorem}\label{c}
	For any $\delta \in\left[e^{-c N}, 1 / 2\right]$, there exists an estimator $\widehat{\Sigma}_\delta\in S(\epsilon)$ such that	$$
	\left\|\Sigma-\widehat{\Sigma}_\delta\right\| \lesssim \sigma\left(\sqrt{\frac{\mathrm{VC}\left(\mathcal{B}_0^*\right)}{N}}+\sqrt{\frac{\log (1 / \delta)}{N}}+\sqrt{\varepsilon}\right) ,
	$$
	where $ \sigma^2=\sup _{u \in \mathcal{B}_2} \mathbb{E}\left(\left\langle u,\left(\Sigma-Y_1 Y_1^T\right) u\right\rangle^2\right)<\infty $.
\end{theorem}
\begin{remark}
	Compared to the similar estimator in \cite{10.1214/19-AOS1862}, our estimator does not require a two-step estimation for the trace and truncation level. In addition, it imposes fewer assumptions. For example, \(L\)-sub-Gaussian or \(L_{4}-L_{2}\) norm equivalence for the sample distribution is no longer necessary.
\end{remark}

Similarly, we can immediately obtain expressions regarding the Frobenius norm as follows:
\begin{corollary}
	For any $\delta \in\left[e^{-c N}, 1 / 2\right]$, there exists an estimator $\widehat{\Sigma}_\delta$ such that	$$
	\left\|\Sigma-\widehat{\Sigma}_\delta\right\|_{F} \lesssim \sigma\left(\sqrt{\frac{\mathrm{VC}\left(\mathcal{B}_0^*\right)}{N}}+\sqrt{\frac{\log (1 / \delta)}{N}}+\sqrt{\varepsilon}\right) ,
	$$
	where $ \sigma $ is the same as in Theorem \ref{c}.
\end{corollary}
Another immediate corollary of Theorem \ref{c} is the quantitative result for the performance of PCA based on the estimator $\widehat{\Sigma}_\delta$. Let $\operatorname{Proj}_k$ be the orthogonal projector on a subspace corresponding to the $k$ largest positive eigenvalues $\lambda_1, \ldots, \lambda_k$ of $\Sigma$ (here, we assume for simplicity that all the eigenvalues are distinct) and $\widehat{\operatorname{Proj}_k}$ be the orthogonal projector of the same rank as $\operatorname{Proj}_k$ corresponding to the $k$ largest eigenvalues of $\widehat{\Sigma}_\delta$. The following bound follows from the Davis-Kahan perturbation theorem in \cite{DBLP:conf/nips/ZwaldB05}.

\begin{corollary}
	Let $\Delta_k=\lambda_k-\lambda_{k+1}$, and assume that $\Delta_k \geq 16 \sigma \sqrt{\frac{K}{N}}$. Then
	$$
	\left\|\widehat{\operatorname{Proj}_k}-\operatorname{Proj}_k\right\| \leq \frac{8}{\Delta_k} \sigma \sqrt{\frac{K}{N}},
	$$
	with probability at least $\left(1-\exp (-K / 128)\right)$.
\end{corollary}

\section{Tukey's median}\label{5}
We consider a special type of median of mean, where we take Tukey's median (see \cite{aloupis2006geometric}), as a robust estimator. First, we need to introduce halfspace depth function. For any $\eta \in \mathbb{R}^d$ and a distribution $\mathbb{P}$ on $\mathbb{R}^d$, the halfspace depth of $\eta$ with respect to $\mathbb{P}$ is defined as
$$
\mathcal{D}(\eta, \mathbb{P})=\inf _{u \in S^{d-1}} \mathbb{P}\left\{u^T X \leq u^T \eta\right\} \quad \text { where } X \sim \mathbb{P} .
$$

Given i.i.d. observations $\left\{X_i\right\}_{i=1}^N$, the halfspace depth of $\eta$ with respect to observations $\left\{X_i\right\}_{i=1}^N$ is defined as
$$
\mathcal{D}\left(\eta,\left\{X_i\right\}_{i=1}^N\right)=\mathcal{D}\left(\eta, \mathbb{P}_N\right)=\min _{u \in S^{p-1}} \frac{1}{N}\sum_{i=1}^N \mathbf{1}_{\left\{u^T X_i \leq u^T \eta\right\}},
$$
where $\mathbb{P}_N=\frac{1}{N} \sum_{i=1}^N \delta_{X_i}$ is the empirical distribution. Then Tukey's median is defined as the deepest point with respect to the observations, that is,
$$
\hat{\theta}=\arg \max _{\eta \in \mathbb{R}^d} \mathcal{D}\left(\eta,\left\{X_i\right\}_{i=1}^N\right) .
$$
\subsection{Halfspace depth}
Define half space $H_{u, \eta}=\left\{y: u^T y \leq u^T \eta\right\}$. Recall that the Tukey's depth of $\eta$ with respect to $\mathbb{P}$ and its empirical counterpart are

$$
\begin{aligned}
\mathcal{D}\left(\eta, \mathbb{P}\right) & =\inf _{u \in S^{p-1}} \mathbb{P}\left(H_{u, \eta}\right)=\inf _{u \in S^{p-1}} \mathbb{P}\left\{u^T Y \leq u^T \eta\right\}, \\
\mathcal{D}\left(\eta,\left\{X_i\right\}_{i=1}^N\right) & =\inf _{u \in S^{p-1}} \mathbb{P}_{N}\left(H_{u, \eta}\right)=\min _{u \in S^{p-1}} \frac{1}{n_1} \sum_{i=1}^{N} \mathbf{1}_{\left\{u^T X_i \leq u^T \eta\right\}}.
\end{aligned}
$$
 The class of set functions $\left\{\mathbf{1}_{H_{u, \eta}}: u \in S^{d-1}, \eta \in \mathbb{R}^d\right\}$ consists of all half spaces in $\mathbb{R}^d$ and hence has VC dimension $d+1$.

In the matter of the coherence with the preceding expressions, we consider the use of its equivalent definition\begin{equation}\label{t}
\hat{\theta}=\arg \min _{\eta \in \mathbb{R}^d}\sup_{v \in \mathcal{B}_0^*}\sum_{i=1}^{N}\mathbf{1}_{\left\langle X_i-\eta,v\right\rangle > 0},
\end{equation}
when (\ref{t}) has multiple maxima, $\hat{\theta}$ is understood as any vector that attains the deepest level. As is known from \cite{rousseeuw1999depth}, the maximum depth $ N\mathcal{D}\left(\hat{\theta},\left\{X_i\right\}_{i=1}^N\right) $ is bounded below by $ \lceil N /(d+1)\rceil $. Because of the natural form of boolean functions, it suggests us to use the VC-dimension technique. From now on, we take our estimator, the Tukey median of mean, as\[\hat{\mu}=\arg \min _{\eta \in \mathbb{R}^d}\sup_{v \in \mathcal{B}_0^*}\sum_{k=1}^{K}\mathbf{1}_{\left\langle \bar{X}_k-\eta,v\right\rangle > 0}.\]  Then we obtain the following error bound:
\begin{theorem}\label{ttt}
	For $ \varepsilon $-corruption samples from an unknown distribution $ \mathbb{P} $, with only finite second moment, there exists a universival constant $C$ such that if $K \geq C\left(\mathrm{VC}\left(\mathcal{B}_0^*\right) \vee|\mathcal{O}|\right)$, then, with probability larger than $1-\exp (-CK)$,
	$$
	\left\|\hat{\mu}_K-\mu\right\| \leq C \sup _{v \in \mathcal{B}_0^*} \mathbb{E}\left(\left\langle Y_1-\mu, v\right\rangle^2\right)^{1 / 2} \sqrt{\frac{K}{N}} .
	$$
\end{theorem}
In particular, if we use the Euclidean distance, then for any $\delta \in\left[e^{-c N}, 1 / 2\right]$, there exists an estimator $\hat{\mu}$ such that with probability at least $ 1-\delta $,
\begin{equation}\label{tt}
\left\|\hat{\mu}-\mu\right\|_{2} \lesssim\left\|\Sigma^{1 / 2}\right\| \left(\sqrt{\frac{d}{N}}+\sqrt{\frac{\log (1 / \delta)}{N}}+\sqrt{\varepsilon}\right).
\end{equation}

\begin{remark}
By adapting Tukey's median to median-of-means method, we extend the robust estimation to heavy tailed settings. In contrast to  Chen et. al. \cite{chen2018robust}, who demonstrated that Tukey's median is effective only in Gaussian and elliptical distributions, our approach broadens its applicability..
\end{remark}

\section{Numerical Studies}
Despite its theoretical appeal, the MOM estimation method proposed by Lugosi and Mendelson, is often computationally intractable. While polynomial-time algorithms exist for specific cases like the mean estimator by Hopkins \cite{hopkins2020mean}, general solutions remain challenging due to exponential growth in complexity, making them impractical for larger datasets. We use the approximation algorithm proposed by \cite{pmlr-v125-lei20a} and write as \textbf{aMOM}. 

Because our implementation is an approximate adaptation (not a literal reproduction of every constant), we adopt the theory-backed scaling $K=\Theta(\log(1/\delta))$ and enforce $K\le N/2$. The spectral inner loop operates on the $K$ bucket means, and its runtime scales roughly as $O(K^2 d)$. Therefore, one can see:
\begin{itemize}
	\item Larger $K$: stronger theoretical robustness, but higher runtime and potentially noisier per-bucket means if $K$ becomes too large relative to $n$.
	\item Smaller $K$: faster in practice, but with weaker formal guarantees (the failure probability bound may no longer match the target $\delta$).
\end{itemize} In practice, one can set
\[
K \;=\; \min\!\big(\,\lceil 500 \log(1/\delta)\rceil,\; \lfloor N/2 \rfloor\,\big)
\]
using a smaller constant $C$ (e.g., less than 500)  to reduce runtime; this typically works well empirically but relaxes the strict theoretical guarantee.  

The computation of depth-based estimators is particularly challenging. While some progress has been made in low-dimensional settings, such as the study on the bivariate Tukey's median \cite{d27883af-202c-34ce-a5f7-a9f82133dba0}, the development of efficient algorithms is fundamentally constrained. The optimal time complexity is $ O(N^d) $ for higher dimension $ d $ according to \cite{nd}. Meanwhile, our theoretical framework suggests natural extensions to covariance estimation and principal component analysis, the computational implementation of these extensions presents significant challenges that require further investigation. We are currently developing efficient algorithms for MOM-based covariance estimation and plan to address these applications in future work.

Before illustrating the setting of experiments, we introduce some robust estimators. The core idea behind Huber’s M-estimator \cite{robust} is to minimize the huber loss function of the residuals,  rather than the sum of squared residuals. The MCD estimator identifies a subset of observations whose sample covariance matrix has the smallest determinant. The Fast MCD algorithm \cite{lirias356358} was developed to make this computationally feasible for larger datasets. MCD can be computed by the package \textsf{MASS}.

\begin{figure}[htbp!]
	\centering
	\includegraphics[width=1\linewidth]{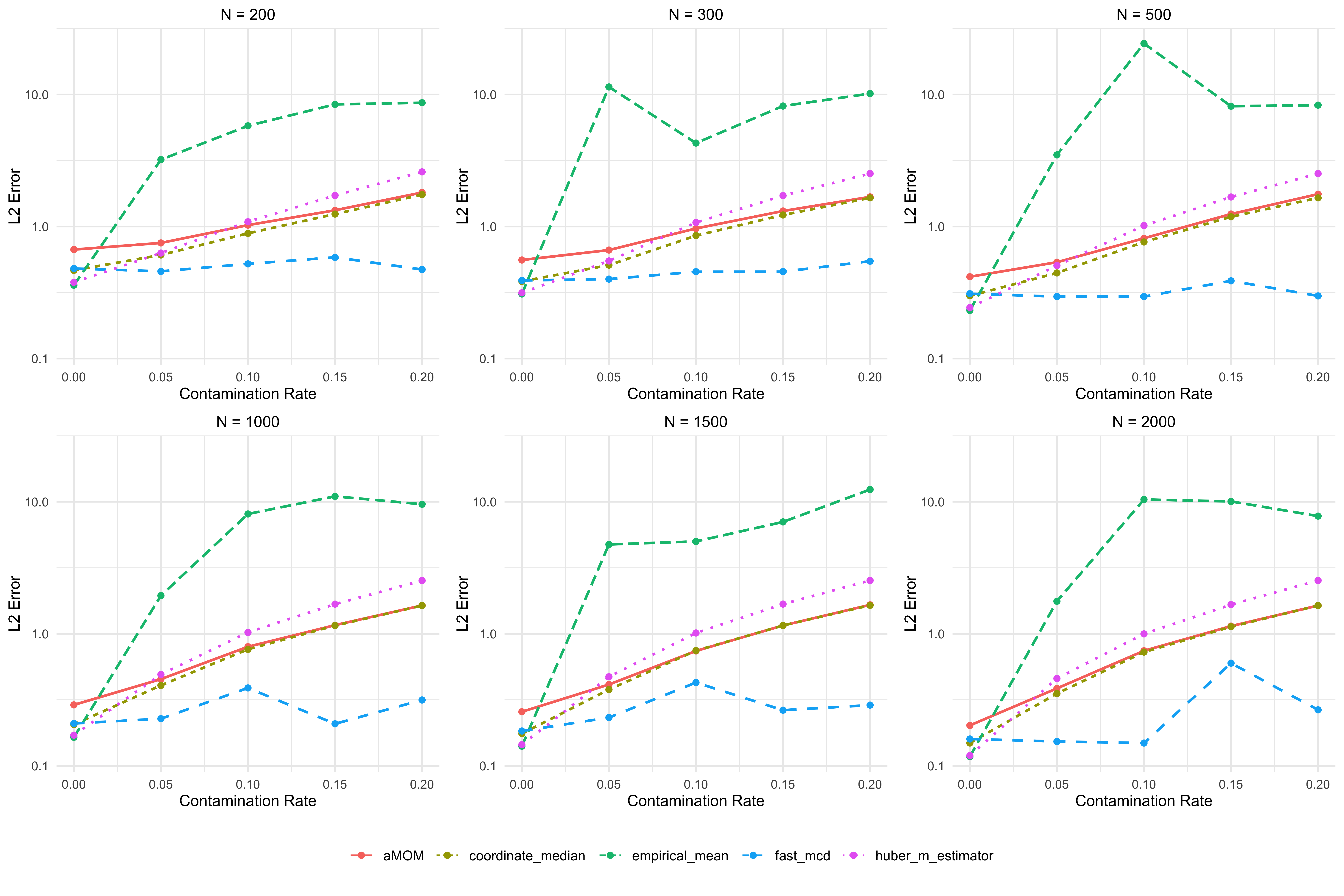}
	\caption{Logarithmic $\ell_2$ estimation error versus contamination proportion ($\varepsilon$)
		for \textbf{Setting 1}: Gaussian inliers with heavy-tailed Cauchy outliers.
		Each curve corresponds to a different estimator: Empirical Mean, Coordinate-wise Median,
		Huber’s M-estimator, Minimum Covariance Determinant (MCD), and the proposed Median-of-Means (MOM) estimator.
		Sample sizes $N \in \{200, 300, 500, 1000, 1500, 2000\}$ with fixed dimension $d = 10$ are used.
		This plot demonstrates that the MOM estimator remains robust even as the contamination level increases,
		maintaining significantly lower estimation error growth compared to classical estimators. }
	\label{fig:0}
\end{figure}

\begin{table}[htbp!]
	\centering
	\caption{Estimation errors for Setting 1 when $\varepsilon = 0$.}
	\label{t1}
	\begin{tabular}{ccccccccc}
		\hline
		$N$ & Empirical Mean & Coordinate Median & Huber's M & MCD & MOM \\
		\hline 
		200&		0.36&	0.47&	0.38&	0.48&	0.67\\
		300&		0.31&	0.38&	0.31&	0.39&	0.56\\
		500&		0.23&	0.3	&	0.24&	0.31&	0.42\\
		1000&	0.16&	0.21&	0.17&	0.21&	0.29\\
		1500&	0.14&	0.18&	0.14&	0.18&	0.26\\
		2000&	0.12&	0.15&	0.12&	0.16&	0.2\\
		\hline
	\end{tabular}
\end{table}

\begin{figure}[htbp!]
	\centering
	\includegraphics[width=1\linewidth]{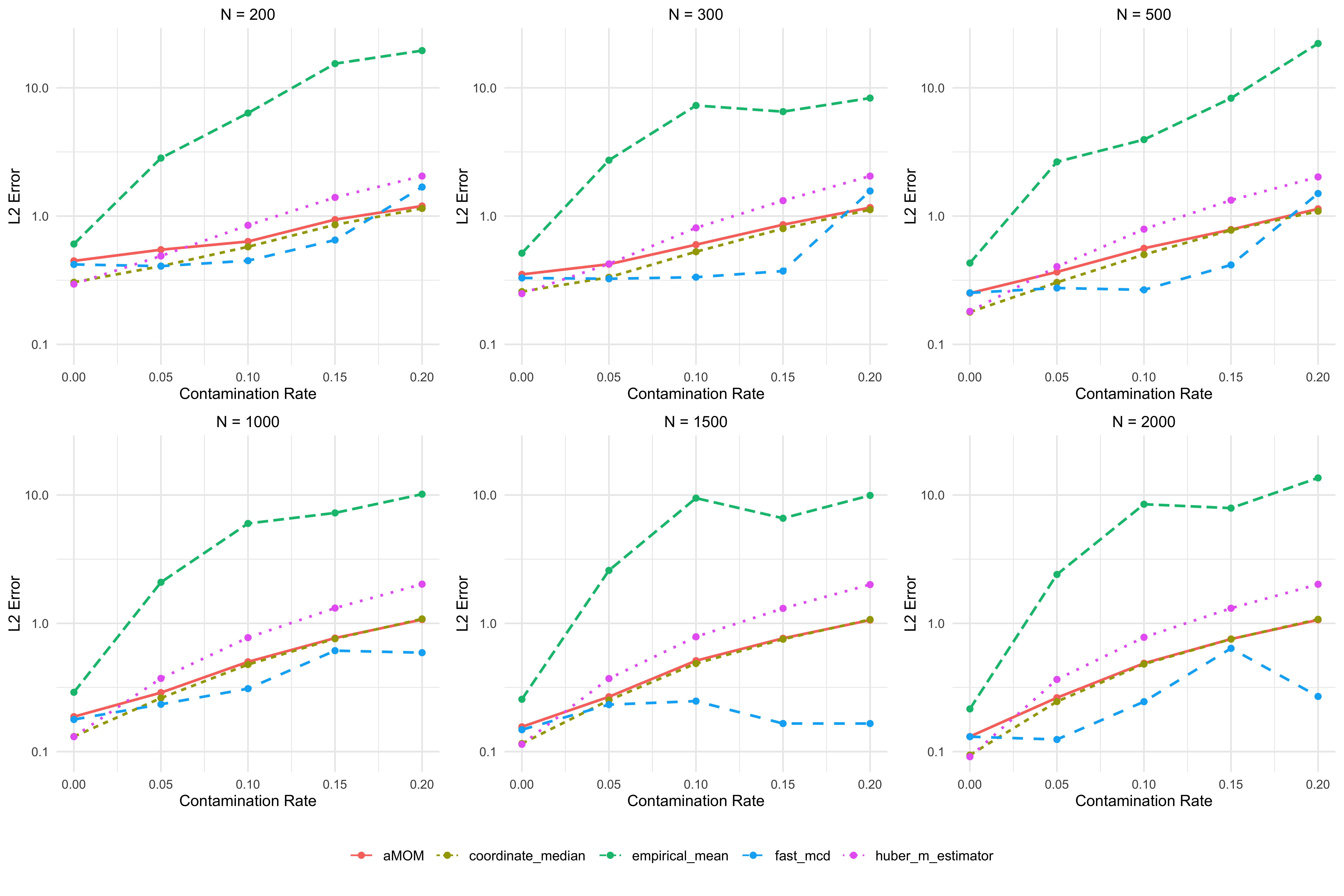}
	\caption{Logarithmic $\ell_2$ estimation error versus contamination proportion ($\varepsilon$)
		for \textbf{Setting 2}: Heavy-tailed Student’s $t_{\nu=2.1}$ inliers with heavy-tailed Cauchy outliers.
		Estimator performance is compared across the same five methods as in Setting 1.
		The experiment highlights the ability of MOM to accurately estimate the mean
		under simultaneous heavy-tailed inlier and outlier distributions,
		exhibiting strong robustness and stable error scaling.}
	\label{fig:1}
\end{figure}

\begin{table}[htbp!]
	\centering
	\caption{Estimation errors for Setting 2 when $\varepsilon = 0$.}
	\label{t2}
	\begin{tabular}{ccccccccc}
		\hline
		$N$ & Empirical Mean & Coordinate Median & Huber's M & MCD & MOM \\
		\hline 
		200&		0.61&	0.3&	0.29&	0.42&	0.45\\
		300&		0.51&	0.26&	0.25&	0.33&	0.35\\
		500&		0.43&	0.18&	0.18&	0.25&	0.25\\
		1000&		0.29&	0.13&	0.13&	0.18&	0.19\\
		1500&		0.26&	0.12&	0.11&	0.15&	0.16\\
		2000	&	0.22&	0.09&	0.09&	0.13&	0.13\\
		\hline
	\end{tabular}
\end{table}
\begin{figure}[htbp]
	\centering
	\includegraphics[width=1\linewidth]{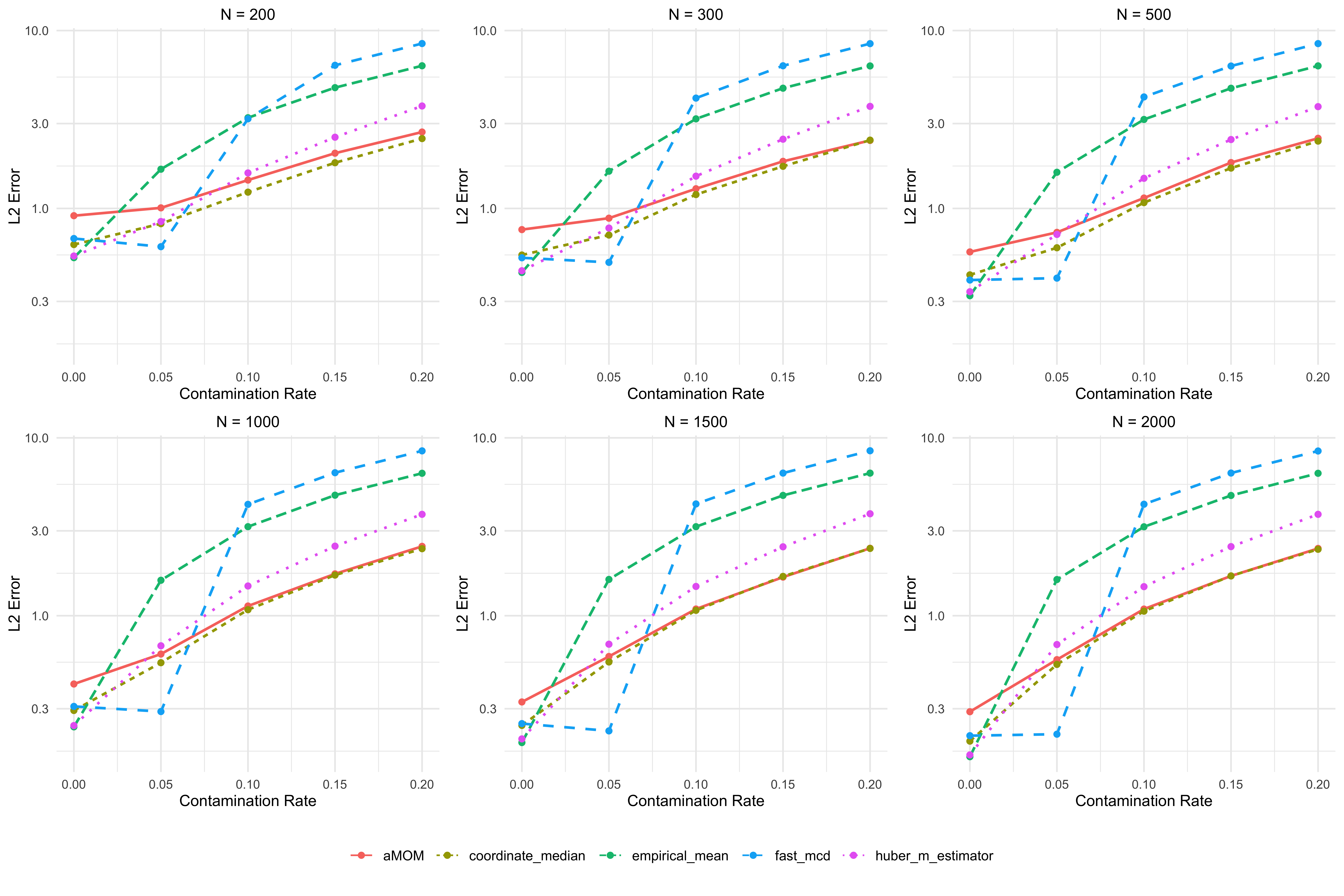}
	\caption{Logarithmic $\ell_2$ estimation error versus contamination proportion ($\varepsilon$)
		for \textbf{Setting 3}: Heteroscedastic Gaussian inliers $N(0, \Sigma)$ with diagonal covariance
		$\Sigma = \text{diag}(1,2,\dots,d)$ and point-mass outliers at $10\cdot \mathbf{1}_d$.
		This experiment evaluates estimator performance under variance heterogeneity and extreme outliers.
		The MOM estimator continues to display robust behavior, adapting to non-identical variances across dimensions
		while providing stable estimation error under contamination.}
	\label{fig:2}
\end{figure}

\begin{table}[htbp]
	\centering
	\caption{Estimation errors for Setting 3 when $\varepsilon = 0$.}
	\label{t3}
	\begin{tabular}{ccccccccc}
		\hline
		$N$ & Empirical Mean & Coordinate Median & Huber's M & MCD & MOM \\
		\hline 
		200&		0.53&	0.63&	0.54&	0.68&	0.91\\
		300	&	0.44&	0.55&	0.45&	0.53&	0.76\\
		500	&	0.32&	0.42&	0.34&	0.4&	0.57\\
		1000	&	0.24&	0.29&	0.24&	0.31&	0.41\\
		1500&		0.19&	0.24&	0.2&	0.25&	0.33\\
		2000	&	0.16&	0.2&	0.16&	0.21&	0.29\\
		\hline
	\end{tabular}
\end{table}
\begin{figure}[htbp]
	\centering
	\includegraphics[width=1\linewidth]{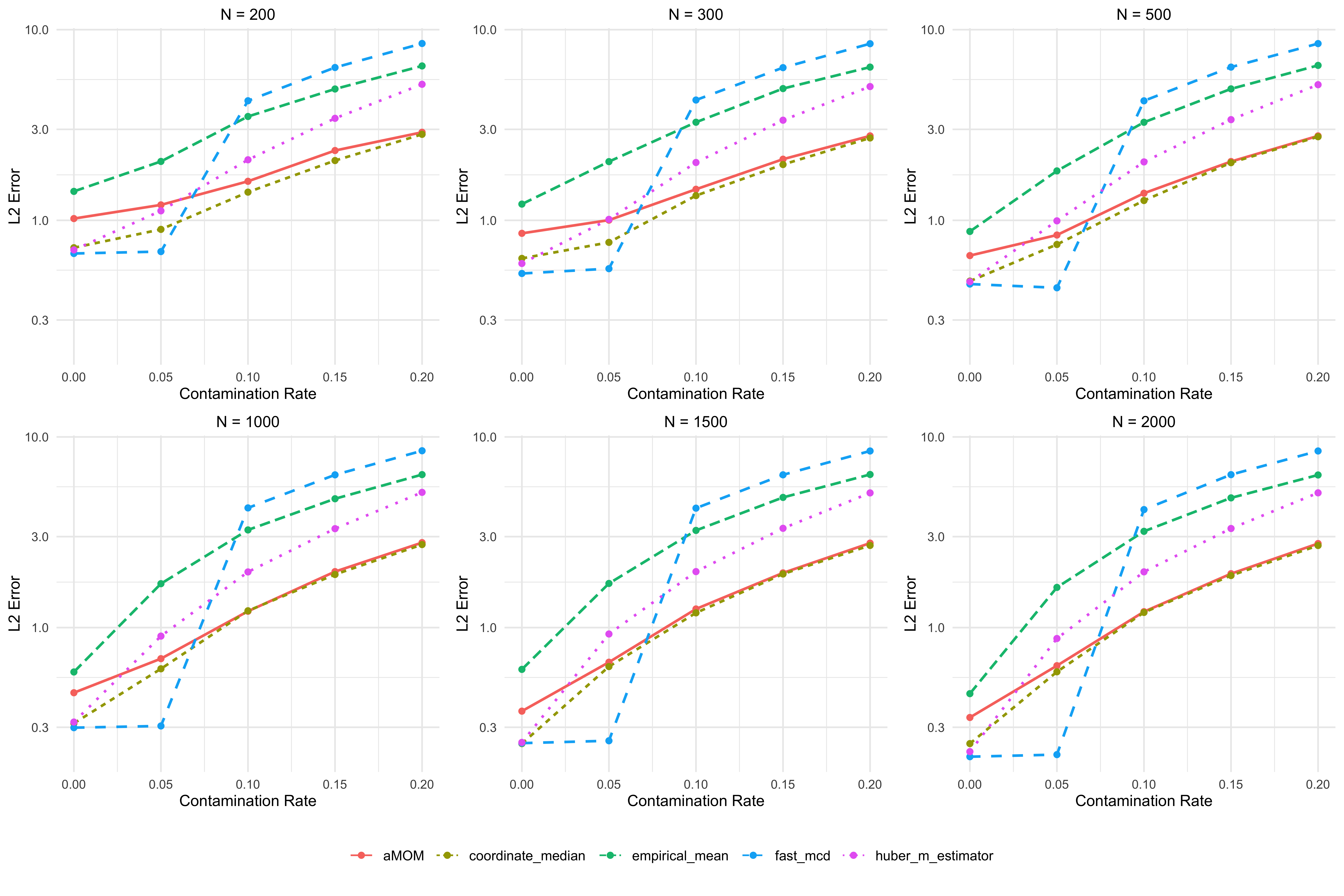}
	\caption{Logarithmic $\ell_2$ estimation error versus contamination proportion ($\varepsilon$)
		for \textbf{Setting 4}: Heteroscedastic heavy-tailed Student’s $t_{\nu=2.1}$ inliers
		with diagonal scale matrix $\Sigma = \text{diag}(1,2,\dots,d)$ and point-mass outliers at $10\cdot \mathbf{1}_d$.
		This more challenging scenario combines heavy tails, heteroscedasticity, and adversarial outliers.
		The MOM estimator remains competitive, showing controlled error growth even under simultaneous multiple difficulties.}
	\label{fig:3}
\end{figure}

\begin{table}[htbp]
	\centering
	\caption{Estimation errors for Setting 4 when $\varepsilon = 0$.}
	\label{t4}
	\begin{tabular}{ccccccccc}
		\hline
		$N$ & Empirical Mean & Coordinate Median & Huber's M & MCD & MOM \\
		\hline 
		200	&	1.42&	0.72&	0.7&	0.67&	1.02\\
		300	&	1.22&	0.63&	0.59&	0.53&	0.85\\
		500	&	0.87&	0.48&	0.48&	0.46&	0.65\\
		1000&	0.59&	0.31&	0.32&	0.3&	0.46\\
		1500	&	0.6&	0.25&	0.25&	0.25&	0.36\\
		2000	&	0.45&	0.25&	0.22&	0.21&	0.34\\
		\hline
	\end{tabular}
\end{table}
\begin{figure}[htbp]
	\centering
	\includegraphics[width=1\linewidth]{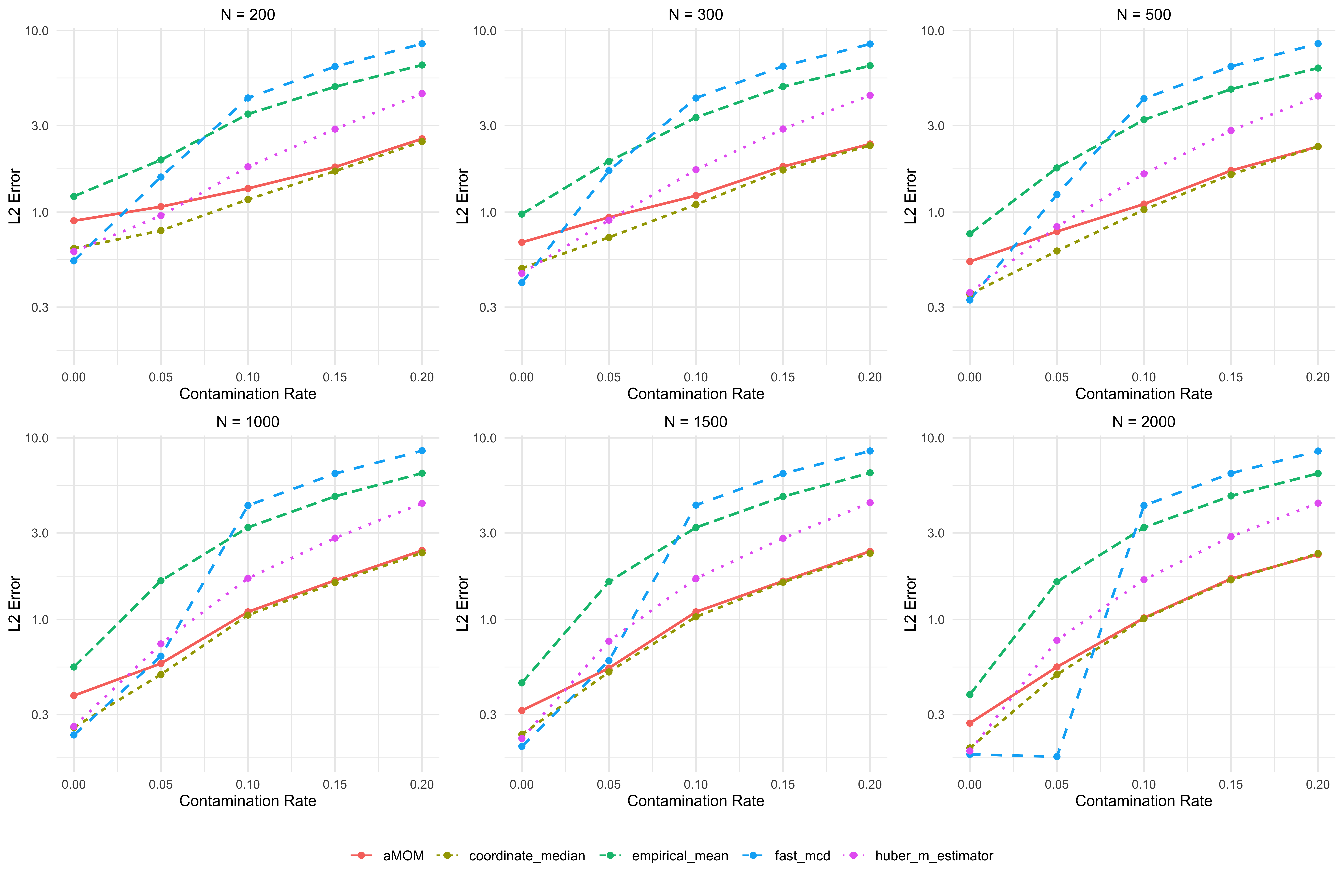}
	\caption{Logarithmic $\ell_2$ estimation error versus contamination proportion ($\varepsilon$)
		for \textbf{Setting 5}: Correlated heavy-tailed Student’s $t_{\nu=3}$ inliers with autoregressive (AR(1)) covariance
		structure $\Sigma_{ij} = 4 \cdot (0.5)^{|i-j|}$ and point-mass outliers at $10\cdot \mathbf{1}_d$.
		This setting evaluates estimator robustness in the presence of correlation, heavy tails, and strong contamination.
		The MOM estimator consistently achieves low estimation error and demonstrates strong resistance to distributional complexity.
	}
	\label{fig:4}
\end{figure}

\begin{table}[htbp]
	\centering
	\caption{Estimation errors for Setting 5 when $\varepsilon = 0$.}
	\label{t5}
	\begin{tabular}{ccccccccc}
		\hline
		$N$ & Empirical Mean & Coordinate Median & Huber's M & MCD & MOM \\
		\hline 
		200	&	1.22&	0.63&	0.61&	0.54&	0.9\\
		300	&	0.98&	0.49&	0.46&	0.41&	0.68\\
		500	&	0.76&	0.35&	0.36&	0.33&	0.54\\
		1000&		0.55&	0.25&	0.26&	0.23&	0.38\\
		1500&		0.45&	0.23&	0.22&	0.2&	0.32\\
		2000	&	0.39&	0.2&	0.19&	0.18&	0.27\\
		\hline
	\end{tabular}
\end{table}

\subsection{Simulation Settings}
\label{sec:simulations}

To evaluate the performance of the proposed estimator, we conduct a series of Monte Carlo simulations under various data-generating processes. In each simulation, the dataset is generated from a contamination model. Let $\mathcal{P}_0$ be the distribution of the clean (inlier) data and $\mathcal{Q}$ be the distribution of the contaminating (outlier) data. The observed dataset $\mathcal{X} = \{X_1, \dots, X_N\}$ of size $N$ consists of $n_{\text{clean}}$ samples drawn from $\mathcal{P}_0$ and $n_{\text{contaminated}}$ samples drawn from $\mathcal{Q}$. The contamination proportion is denoted by $\varepsilon = n_{\text{contaminated}} / N$.

Our numerical experiments are configured with the following parameters: sample sizes $ N \in \{200, 300, 500, 1000, 1500, 2000\} $, a fixed dimension $ d = 10 $, and contamination proportions $ \varepsilon\in\{0, 0.05,0.1,0.15, 0.2\}$.  The estimation error, measured by the $ l_2 $-norm, is averaged over 50 independent trials. 

All simulations are performed in a $d$-dimensional space. The true mean of the clean data is the zero vector, i.e., $\mu_0 = \mathbf{0}_d \in \mathbb{R}^d$. We consider five distinct settings designed to test the estimators' robustness against different types of distributions and contamination schemes. 
Moreover, for the case $ \varepsilon=0 $, we also demonstrate the statistical efficiency of these robust estimators in Table \ref{t1}-\ref{t5}.

\begin{itemize}
	\item[\textbf{Setting 1:}] \textbf{Gaussian Inliers with Heavy-Tailed Outliers.}
	This is a canonical setting for robust estimation. The inlier distribution $\mathcal{P}_0$ is a multivariate normal distribution with a spherical covariance structure:
	\[
	\mathcal{P}_0 = \mathcal{N}(\mathbf{0}_d, 3\mathbf{I}_d),
	\]
	where $\mathbf{I}_d$ is the $d \times d$ identity matrix. The contamination distribution $\mathcal{Q}$ generates outliers from a multivariate Cauchy distribution (a t-distribution with 1 degree of freedom) shifted far from the origin. Specifically, an outlier $X_{\text{out}} \sim \mathcal{Q}$ is generated as $X_{\text{out}} = Z + 10 \cdot \mathbf{1}_d$, where each component of $Z \in \mathbb{R}^d$ is drawn independently from a standard Cauchy distribution, and $\mathbf{1}_d$ is a vector of ones.
	
	\item[\textbf{Setting 2:}] \textbf{Heavy-Tailed Inliers with Heavy-Tailed Outliers.}
	This setting tests the estimator's ability to distinguish between heavy-tailed inliers and outliers. The inlier distribution $\mathcal{P}_0$ is heavy-tailed, where each coordinate is independently drawn from a Student's t-distribution with $\nu=2.1$ degrees of freedom, centered at the origin. The contamination distribution $\mathcal{Q}$ is identical to that in Setting 1 (shifted Cauchy distribution).
	
	\item[\textbf{Setting 3:}] \textbf{Heteroscedastic Gaussian Inliers with Point-Mass Outliers.}
	Here, we introduce non-identical variances across dimensions for the inlier data. The inlier distribution $\mathcal{P}_0$ is a multivariate normal distribution with a diagonal, heteroscedastic covariance matrix:
	\[
	\mathcal{P}_0 = \mathcal{N}(\mathbf{0}_d, \Sigma), \quad \text{where } \Sigma = \text{diag}(1, 2, \dots, d).
	\]
	The contamination $\mathcal{Q}$ is a point-mass distribution, where all outliers are fixed at a single point far from the true mean: $\mu_{\text{cont}} = 10 \cdot \mathbf{1}_d$.
	
	\item[\textbf{Setting 4:}] \textbf{Heteroscedastic Heavy-Tailed Inliers with Point-Mass Outliers.}
	This setting combines the challenges of heavy-tailed inliers and heteroscedasticity. The inlier distribution $\mathcal{P}_0$ is a multivariate Student's t-distribution with $\nu=2.1$ degrees of freedom, centered at the origin, and a diagonal scale matrix $\Sigma = \text{diag}(1, 2, \dots, d)$. The contamination distribution $\mathcal{Q}$ is the same point-mass distribution as in Setting 3, located at $10 \cdot \mathbf{1}_d$.
	
	\item[\textbf{Setting 5:}] \textbf{Correlated Heavy-Tailed Inliers with Point-Mass Outliers.}
	Finally, we investigate the effect of correlated features in the inlier data. The inlier distribution $\mathcal{P}_0$ is a multivariate Student's t-distribution ($t_{\nu}(\mu, \Sigma)$) with $\nu=3$ degrees of freedom and centered at $\mu_0 = \mathbf{0}_d$. The scale matrix $\Sigma$ has an autoregressive (AR(1)) structure, where the entry $(i, j)$ is given by:
	\[
	\Sigma_{ij} = 4 \cdot (0.5)^{|i-j|}.
	\]
	This structure introduces a moderate correlation between adjacent dimensions. The contamination distribution $\mathcal{Q}$ is again the point-mass distribution located at $10 \cdot \mathbf{1}_d$.
\end{itemize}

\subsection{High-Dimensional Performance}
Although high-dimensional scenarios are not specifically addressed in our theoretical framework, we now present additional Monte Carlo simulations to this end, involving 
$ d=100 $
and 
$ d=200 $
, which are conducted under \textbf{Setting 2} (i.e., heavy-tailed inliers with heavy-tailed outliers).  The estimation error, measured by the $ l_2 $-norm, is averaged over 50 independent trials, and other settings are the same as when $ d=10 $.

\begin{figure}[htbp]
	\centering
	\includegraphics[width=1\linewidth]{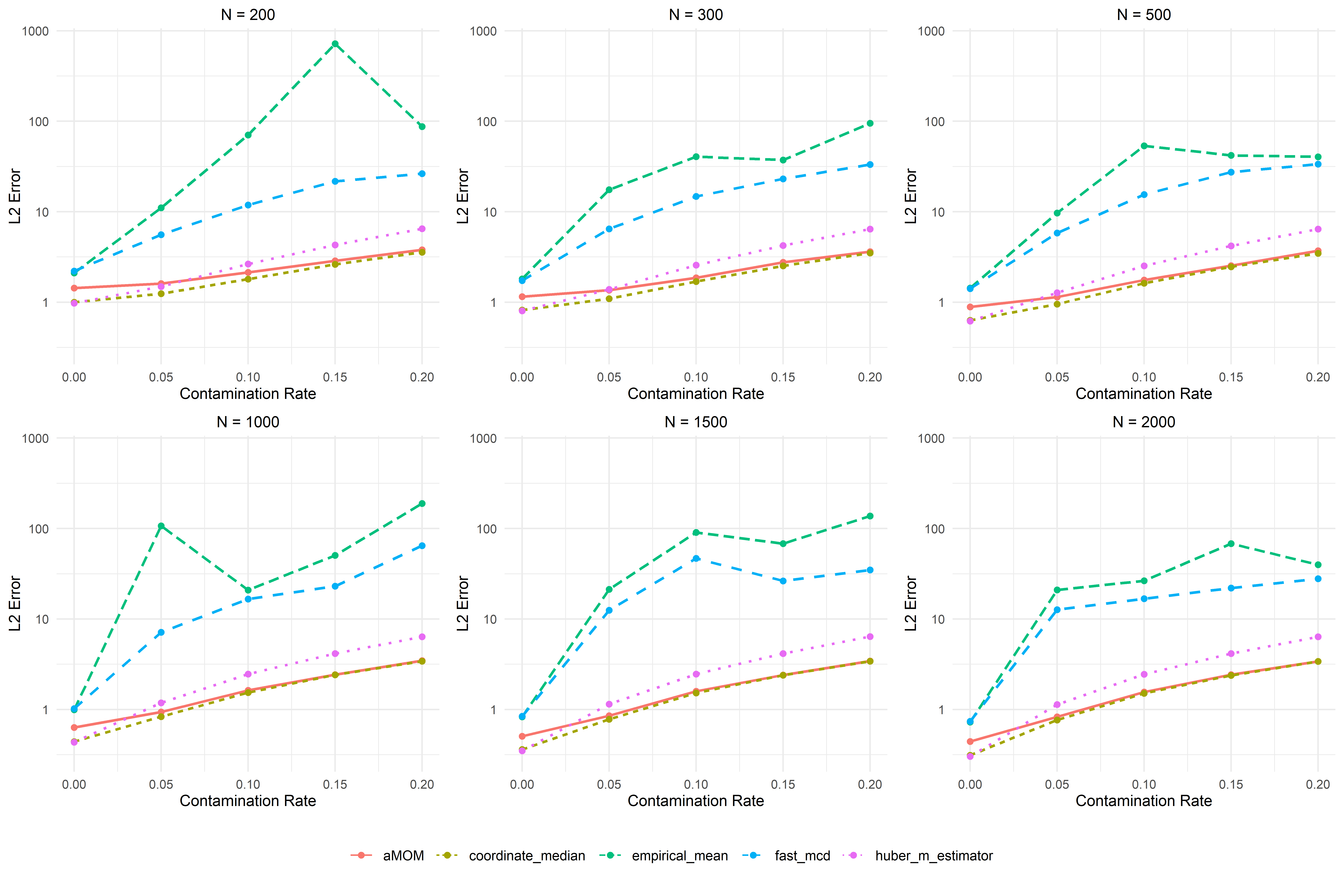}
	\caption{Logarithmic $\ell_2$ estimation error versus contamination proportion ($\varepsilon$)
		for \textbf{Setting 2}: Heavy-tailed Student’s $t_{\nu=2.1}$ inliers with heavy-tailed Cauchy outliers when $ d=100 $.
	}
	\label{fig:100}
\end{figure}

\begin{figure}[htbp]
	\centering
	\includegraphics[width=1\linewidth]{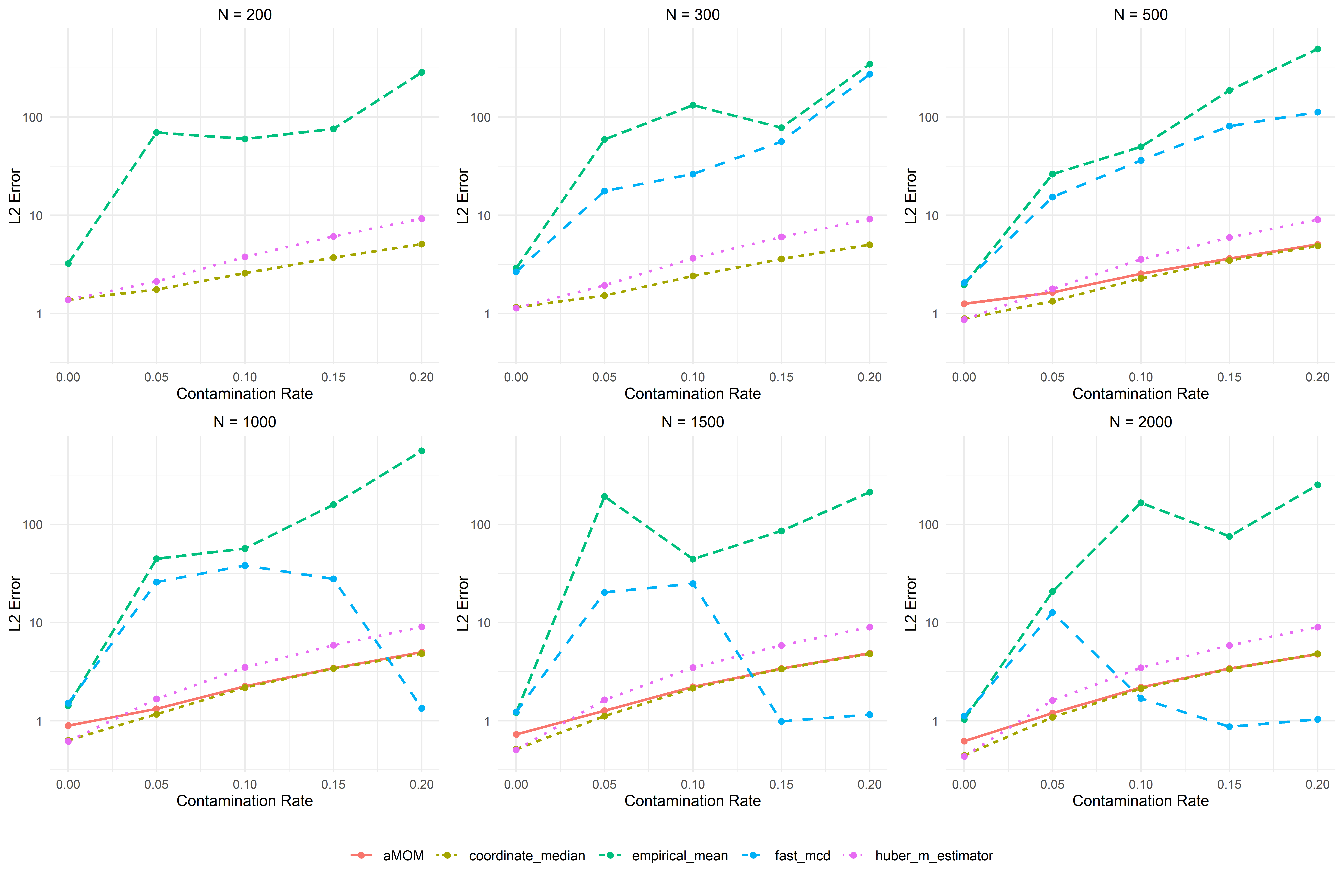}
	\caption{Logarithmic $\ell_2$ estimation error versus contamination proportion ($\varepsilon$)
		for \textbf{Setting 2}: Heavy-tailed Student’s $t_{\nu=2.1}$ inliers with heavy-tailed Cauchy outliers when $ d=200 $.
	}
	\label{fig:200}
\end{figure}

The simulation results, as illustrated in the figures and tables, provide a clear comparison of the estimators' performance under various challenging conditions. Our key findings are as follows:

The empirical mean is shown to be highly non-robust and entirely unsuitable for mean estimation in the presence of heavy-tailed distributions and data contamination. Its performance degrades catastrophically even with a small fraction of outliers. The MCD estimator performs well under specific distributional assumptions. However, it is highly sensitive to the proportion of contamination and dimension. As the contamination level approaches its breakdown point, the accuracy of the MCD estimator deteriorates sharply, indicating its vulnerability in heavily contaminated scenarios.

In contrast, the MOM estimator demonstrates strong and consistent robustness. It reliably handles both heavy-tailed data and a significant fraction of outliers, maintaining low estimation error across most of our simulated settings, even when the  dimension $ d $ is high. 

\subsection{Real Data Analysis}

We utilized a multi-dimensional financial dataset sourced from the Federal Reserve Economic Data (FRED) database (available in R package \textsf{quantmod}) for real world data. The dataset comprises 12 time-series variables, including U.S. Treasury constant maturity rates and the U.S. Dollar to Euro exchange rate, spanning from January 2020 to December 2024.

To ensure stationarity, a preprocessing protocol, including the removal of missing values, was applied. The resulting clean dataset consists of $N=1245$ observations and $d=12$ dimensions. The inter-variable correlation structure is non-trivial, with an average correlation coefficient of 0.432 and a range from -0.258 to 0.979, reflecting typical characteristics of financial markets. To simulate the presence of outliers, a contaminated dataset was generated: a fraction $\varepsilon = 0.1$ of the samples were randomly selected or contamination. For each selected sample, a subset of its dimensions (between 1 and 5, chosen uniformly at random) was corrupted. The contamination was additive, with noise drawn from a normal distribution $\mathcal{N}(0, 3^2)$, representing significant but not extreme outliers.

To assess robustness without relying on a known ground-truth mean, we employed an evaluation framework based on 5-fold cross-validation (CV) of 3 repeats. The CV risk for a given fold was calculated as the median of these squared losses over all points in the test set. The primary metric for comparing estimator robustness is the Robustness Ratio, defined as:
\begin{equation*}
\text{Robustness Ratio} = \frac{\text{Mean CV Risk on Contaminated Data}}{\text{Mean CV Risk on Clean Data}}
\end{equation*}
A ratio close to 1 indicates high robustness, implying that the estimator's predictive performance is minimally affected by data contamination.

\begin{table}[]
	\centering
	\caption{Robustness Ratio of five location estimators. Lower values indicate superior robustness.}
	\label{t6}
	\begin{tabular}{lccccc}
		\hline
		\textbf{Method} & Empirical Mean & Coordinate Median & Huber'M & MCD &MOM  \\
		\hline
		\textbf{Ratio}  & 1.323  & 1.233 & 1.250 & 1.255 &\textbf{1.220} \\
	\end{tabular}
\end{table}
The primary finding of this study is the comparative robustness of the five estimators, as summarized by the Robustness Ratio in Table \ref{t6}.
Our MOM estimator achieved a Robustness Ratio of 1.220, the lowest among all tested methods. This indicates that it is robust when faced with 10\% data contamination, demonstrating superior stability. Other robust estimators show moderate performance. As expected, the empirical mean was highly sensitive to contamination, with its CV risk increasing by 32.3\%, confirming its unsuitability for contaminated data.

\begin{figure}[h!]
	\centering
	\includegraphics[width=\textwidth]{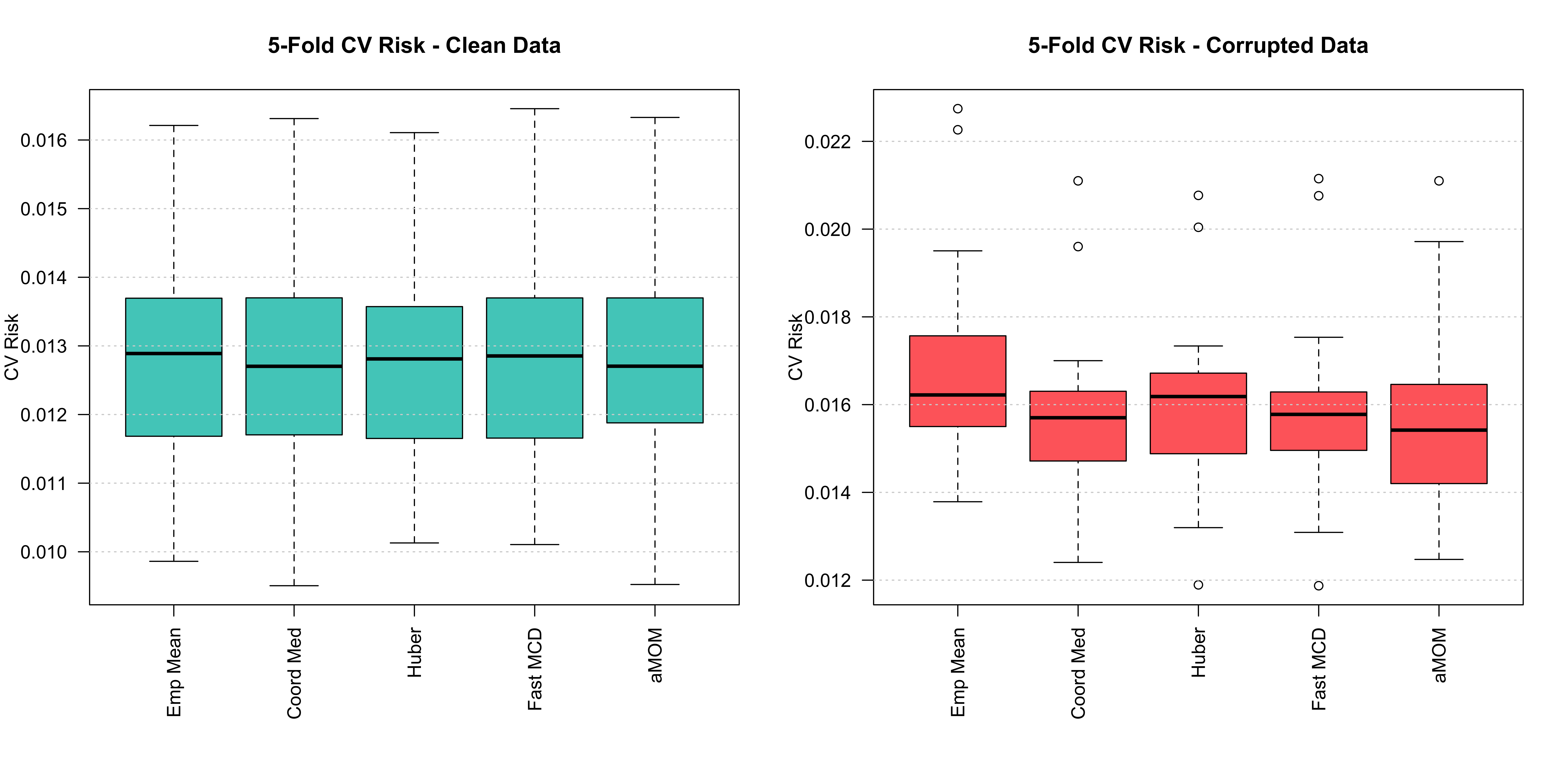}
	\caption{Comparison of 5-fold cross-validated risk distributions for each estimator on Clean Data (left) and Corrupted Data (right). The boxes represent the interquartile range (IQR), the central line is the median, and the whiskers extend to 1.5 times the IQR. Circles denote outliers.}
	\label{fig:risk_boxplots}
\end{figure}
The boxplots in Figure \ref{fig:risk_boxplots} provide deeper insights into the performance distributions:

\begin{itemize}
	\item \textbf{Performance on Clean Data}: On the clean dataset, all five estimators demonstrated nearly identical CV risk distributions, with median values tightly clustered around 0.0128. This is a crucial observation, indicating that the robust estimators do not sacrifice significant efficiency under ideal (uncontaminated) conditions.
	\item \textbf{Performance on Contaminated Data}: On the contaminated dataset, a clear performance separation emerges. The empirical mean exhibits the highest median CV risk and the largest variance, with several high-leverage points. In contrast, the four robust estimators, especially MOM, successfully controlled the increase in risk, showing lower medians.
\end{itemize}

This analysis demonstrates that, in a mulit-dimensional financial data setting, robust estimators offer a significant advantage over the classical empirical mean when data is contaminated. Among the evaluated methods, the MOM estimator emerges as the most robust choice. This suggests its strong potential as a reliable tool for central tendency estimation in financial applications where data quality cannot be guaranteed.

\section{Discussion}
  As noted earlier, the method we present can be implemented in mean estimation, covariance estimation, and other learning problems.
  The problem of regression functions estimation involves estimating conditional expectations, making it a natural extension of the mean estimation concepts discussed in this paper.

  	For instance, the standard framework for regression function estimation is as follows. Consider a pair of random variables $(Y, V)$, where $Y$ takes values in the set $\mathcal{X}$ and $V$ is real-valued. In a class $\mathcal{F}$ comprising real-valued functions defined on $\mathcal{X}$, the goal is to identify $f \in \mathcal{F}$ such that $f(Y)$ serves as a reliable prediction of $V$. The efficacy of a predictor $f \in \mathcal{F}$ is measured through the mean-squared error $\mathbb{E}(f(Y)-V)^2$, known as the risk. The optimal performance within the class is achieved by the risk minimizer	
\begin{equation*}
  	f^*=\underset{f \in \mathcal{F}}{\operatorname{argmin}} \mathbb{E}(f(Y)-V)^2.
\end{equation*}
  	The joint distribution of $(Y, V)$ is usually unknown. Instead, an i.i.d. sample $\mathcal{D}_N=\left(Y_i, V_i\right)_{i=1}^N$ is provided, distributed according to the joint distribution of $Y$ and $V$. Given a sample size $N$, a learning procedure is a mapping assigning to each sample $\mathcal{D}_N$ a function in $\mathcal{F}$, denoted as $\widehat{f}$.
  	
  	The effectiveness of $\widehat{f}$ is assessed based on the trade-off between accuracy $\epsilon$ and confidence $\delta$ in which $\widehat{f}$ achieves that accuracy. In other words, one seeks $\widehat{f}$ that satisfies the condition

\begin{equation*}
  	\mathbb{P}\left(\mathbb{E}\left((\widehat{f}(Y)-V)^2 \mid \mathcal{D}_N\right) \leq \inf _{f \in \mathcal{F}} \mathbb{E}(f(Y)-V)^2+\epsilon\right) \geq 1-\delta
\end{equation*}
for values of $\epsilon$ and $\delta$ as small as possible. The exploration of this accuracy/confidence trade-off has been the focus of extensive research: see \cite{mendelson2017optimal,lecue2017regularization,lecue2018regularization,LM2}.
Consider the standard linear regression setting, where, $ \mathcal{F}=\{\left\langle\beta, Y\right\rangle: \beta\in\mathbb{R}^d\}$, and
	$$
	\beta^*=\underset{\beta \in \mathbb{R}^d}{\operatorname{argmin}} l(\beta)=\underset{\beta \in \mathbb{R}^d}{\operatorname{argmin}} \mathbb{E}\left(V_1-\left\langle\beta, Y_1\right\rangle\right)^2 .
	$$
	Since for all $\beta$,
	$
	l(\beta)-l\left(\beta^*\right)=2 \mathbb{E}\left(\xi_1\left\langle\beta-\beta^*, Y_1\right\rangle\right)+\left(\beta-\beta^*\right) \Sigma\left(\beta-\beta^*\right) \leq \left(\beta-\beta^*\right) \Sigma\left(\beta-\beta^*\right),
	$
	the key to control excess risk is to bound the $ \left\| \beta-\beta^*\right\|_{\Sigma}= \left(\beta-\beta^*\right) \Sigma\left(\beta-\beta^*\right) .$ 

Brownless et. al. \cite{brownlees2015empirical} show that, using the generic chaining method, Catoni estimator can also be linked to the upper bounds of random processes, which can be used to analyze and control the upper bounds of risks in the empirical risk minimization process. 

For practical implementation, we can employ the R package \textsf{TukeyRegion}, which implements the fast algorithm proposed by \cite{lirias356358}, to compute Tukey's median. Through a median-of-means approach, we reduce the computational complexity of the Tukey's median estimator to $ O(K^d) $; nonetheless, it remains computationally intensive, especially in high dimensions.
Chen et. al. \cite{chen2018robust} defines the matrix depth of a positive semidefinite $\Gamma \in \mathbb{R}^{d \times d}$ with respect to a distribution $\mathbb{P}$ and empirical distribution $\mathbb{P}_N$. They also provide a framework for computing the deepest matrix estimator, which serves as an inspiration for us to propose corresponding estimators based on MOM. We hypothesize that the error rate associated with these estimators will be comparable to that of the estimators derived from the matrix depth approach.

Moreover, in the multivariate setting, different definitions of the median lead to different MOM estimators.  Apart from the Tukey median,
many other types of median have been developed, such as coordinate-wise median, the geometric (or spatial) median, Oja median, and Liu median, among others; see \cite{small1990survey} for a survey. 

The introduction of VC dimension has provided new ideas for estimating the upper bounds of empirical processes. It must be reiterated that we have circumvented the Rademacher complexity and, correspondingly, replaced it with the VC dimension term. The critical issue in the proof of bound errors is to find the proper boolean functions class and the use of Lemma \ref{l}.  Consequently, in the estimation of the error bound, the focus is more on the dimensional structure of the set/space, thereby mitigating the impact of heavy tails and contamination in the samples. At the same time, we have only addressed the statistical convergence rates in the estimation problem without considering efficiency concerns in the computation. 

\section*{Acknowledgement}Yuxuan Wang and Lixin Zhang were supported by grants from National Key R\&D Program of China (No. 2024YFA1013502) and  NSF of China  (Grant Nos. U23A2064 and 12031005). Hanchao Wang  was supported by the National Natural Science Foundation of China (No. 12071257 and No. 11971267); National Key R\&D Program of China (No. 2018YFA0703900 and No. 2022YFA1006104); Shandong Provincial Natural Science Foundation (No. ZR2019ZD41).

\section*{Disclosure statement}
The authors declare no conflict of interest.

\appendix
\section{Appendix: Lemmas and Technical Proofs}
There are some basic facts about VC dimension from Section 7 in \cite{sen2018gentle}.
\begin{proposition}
	\begin{enumerate}
		\item $\mathrm{VC}\left(S^{d-1}\right)=d+1$. If $F$ is a set of real-valued functions in a $k$-dimensional linear space, then $\operatorname{Pos}(F):=\left\{x \rightarrow \mathbf{1}_{f(x) \geq 0}, f \in F\right\}$ has VC dimension $k+1$ .
		
		\item For a function $g: \mathcal{Y} \rightarrow \mathcal{X}$, if we note $\mathcal{F} \circ g=\{f \circ g \mid f \in \mathcal{F}\}$, then we have $\operatorname{VC}(\mathcal{F} \circ g) \leq \operatorname{VC}(\mathcal{F})$.
		
		\item  For any $r>0, \mathrm{VC}\left(\left\{x \in E \rightarrow \mathbf{1}_{\langle x, v\rangle \geq r}, v \in C\right\}\right) \leq \mathrm{VC}(C-C) \lesssim \mathrm{VC}(C)$. 
	\end{enumerate}
\end{proposition}

The following lemma can be found in \cite{van2009note}:
\begin{lemma}\label{v}
	If $ \operatorname{VC}(\mathcal{C}_{i})=V_{i} ,$ $ i=1,\dots,m, $ 
	let $V \equiv \sum_{j=1}^m V_j$, and let
	$$\begin{aligned} \sqcup_{j=1}^m \mathcal{C}_j & \equiv\left\{\cup_{j=1}^m C_j: C_j \in \mathcal{C}_j, j=1, \ldots, m\right\}, \\ \sqcap_{j=1}^m \mathcal{C}_j & \equiv\left\{\cap_{j=1}^m C_j: C_j \in \mathcal{C}_j, j=1, \ldots, m\right\},\\\boxtimes_{j=1}^m \mathcal{C}_j &\equiv\left\{C_1 \times \ldots \times C_m: C_j \in \mathcal{C}_j, j=1, \ldots, m\right\}.\end{aligned}$$
	Then the following bounds hold:
	$$
	\left\{\begin{array}{l}
	V\left(\sqcup_{j=1}^m \mathcal{C}_j\right) \\
	V\left(\sqcap_{j=1}^m \mathcal{C}_j\right) \\
	V\left(\boxtimes_1^m \mathcal{C}_j\right)
	\end{array}\right\} \leq c_1 V \log \left(c_2 m\right), \\
	$$
	where $ c_1 = \frac{e}{(e-1) \log (2)} \approx 2.28231. $
\end{lemma}
The following is a classical result of Vapnik-Chervonenkis theory (more details can be found in \cite{vershynin2018high}), which shows the connection between the VC dimension and empirical processes.
\begin{lemma}[Empirical processes via $\mathrm{VC}$ dimension]\label{lc}
	Let $\mathcal{F}$ be a class of Boolean functions in a probability space $(\Omega, \mathcal{A}, \mathbb{P})$ with finite VC dimension $\operatorname{VC}(\mathcal{F}) \geq 1$. Let $X, X_1, X_2, \ldots, X_N$ be independent random points in $\Omega$ distributed according to the law $\mathbb{P}$. Then
	$$
	\mathbb{E} \sup _{f \in \mathcal{F}}\left|\frac{1}{N} \sum_{i=1}^N f\left(X_i\right)-\mathbb{E} f(X)\right| \leq C \sqrt{\frac{\operatorname{VC}(\mathcal{F})}{N}} .
	$$
\end{lemma}

\subsection{Proof of Theorem \ref{m}}
\begin{proof}
	Let $\mathcal{F}=\left\{\left(x_i\right)_{i \leq m} \rightarrow \mathbf{1}_{\left\langle\frac{1}{m} \sum_i x_i -\mu, v\right\rangle \geq r_K}, v \in \mathcal{B}_0^*(\mathcal{U})\right\}$, where$$
	r_K=4 \sup _{v \in \mathcal{B}_0^*(\mathcal{U})} \mathbb{E}\left(\left\langle Y_1-\mu, v\right\rangle^2\right)^{1 / 2} \sqrt{\frac{K}{N}}.
	$$ 
	For any $k \in [K]$, let $\mathbf{X}_k:=\left(X_i\right)_{i \in B_k}$ and $\mathbf{Y}_k:=\left(Y_i\right)_{i \in B_k}$. The functions $f \in \mathcal{F}$ are compositions of the function $x \rightarrow \frac{1}{m} \sum_i x_i-\mu$ and of the functions $x \rightarrow \mathbf{1}_{\langle x, v\rangle \geq r_K}$ for $v \in \mathcal{B}_0^*(\mathcal{U})$. The VC-dimension of the set of these compositions is smaller than the $\mathrm{VC}$-dimension of the set of indicator functions indexed by $\mathcal{B}_0^*(\mathcal{U})$. We just get $\mathrm{VC}(\mathcal{F}) \leq c_0 \mathrm{VC}(\mathcal{B}_0^*(\mathcal{U}))$ for some constant $c_0$.
	
	Notice that, the definition in (\ref{s}) is equivalent to$$S_{x^*}=\left\{y \in \mathcal{U}:\left|x^*\left(\bar{X}_k\right)-x^*(y)\right| \leq \epsilon \text { for more than } \frac{K}{2} \text { blocks }\right\} .$$For any $f\in\mathcal{F}$, or, for the corresponding $v \in \mathcal{B}_0^*(\mathcal{U})$, there exist at least $\left(K-\sum_{k=1}^K f\left(\mathbf{X}_k\right)\right)$ blocks $B_k$, where
	$$
	\left|\left\langle\bar{X}_k-\mu,v\right\rangle\right| \leq r_K.
	$$
	So it is sufficient to compute the sum of $f\left(\mathbf{X}_k\right)$. Now we write
	
	\begin{equation}\label{l3}
	\begin{aligned}
	\sup _{f \in \mathcal{F}} \sum_{k=1}^K f\left(\mathbf{Y}_k\right)=& \left[ \sup _{f \in \mathcal{F}} \sum_{k=1}^K f\left(\mathbf{Y}_k\right)-\mathbb{E}\left(\sup _{f \in \mathcal{F}} \sum_{k=1}^K f\left(\mathbf{Y}_k\right)\right)\right]\\&+\mathbb{E}\left(\sup _{f \in \mathcal{F}} \sum_{k=1}^K f\left(\mathbf{Y}_k\right)\right) .
	\end{aligned}
	\end{equation}
	Let $ f_{j}(\mathbf{Y})=\sup _{f \in \mathcal{F}} \left( \sum_{k\ne j} f\left(\mathbf{Y}_k\right)+f\left(\mathbf{Y}\right)\right)  $, since $ f $ is binary-valued, we have $\left| f_{j}(\mathbf{Y})-f_{j}(\mathbf{Y'})\right|\leq 1 ,$ for any $ j\in [K] $ and $ \mathbf{Y}, \mathbf{Y'}\in \mathbb{R}^{d\times m}$. By the bounded difference inequality (see \cite{boucheron2003concentration}),\[\mathbb{P}\left(\sup _{f \in \mathcal{F}} \sum_{k=1}^K f\left(\mathbf{Y}_k\right)-\mathbb{E}\left(\sup _{f \in \mathcal{F}} \sum_{k=1}^K f\left(\mathbf{Y}_k\right)\right) \geq t\right)  \leq \exp \left(-\frac{2 t^2}{K}\right).\] Therefore, by taking $ t=K/16 $, we can derive that, with probability at least $ 1-\exp (-K/128)$ , the first term in (\ref{l3}) is bounded above by $K/16$.
	
	For the second term, we further write
	$$
	\mathbb{E}\left(\sup _{f \in \mathcal{F}} \sum_{k=1}^K f\left(\mathbf{Y}_k\right)\right) \leq \mathbb{E}\left(\sup _{f \in \mathcal{F}} \sum_{k=1}^K f\left(\mathbf{Y}_k\right)-K \mathbb{E}\left(f\left(\mathbf{Y}_k\right)\right)\right)+\sup _{f \in \mathcal{F}} K \mathbb{E}\left(f\left(\mathbf{Y}_k\right)\right) .
	$$
	By Markov's inequality, for any $v \in \mathcal{B}_0^*(\mathcal{U})$,
	$$
	\mathbb{P}\left(\left|\left\langle\frac{1}{m} \sum_{i \in B_1} Y_i-\mu, v\right\rangle\right| \geq r_K\right) \leq \frac{\mathbb{E}\left(\sum_{i \in B_1}\left\langle Y_i-\mu, v\right\rangle^2\right)}{m^2 r_K^2} \leq \frac{1}{16}.
	$$Then, $\sup _{f \in \mathcal{F}} K \mathbb{E}\left(f\left(\mathbf{Y}_k\right)\right) \leq K/16$. And by Lemma \ref{lc}, there exists a universal constant $C^{\prime}$ such that
	$$
	\mathbb{E}\left(\sup _{f \in \mathcal{F}} \frac{1}{K}\sum_{k=1}^K f\left(\mathbf{Y}_k\right)- \mathbb{E}\left(f\left(\mathbf{Y}_k\right)\right)\right) \leq C^{\prime}  \sqrt{\frac{\mathrm{VC}(\mathcal{F})}{K}} .
	$$
	Hence, if $K \geq 256 C^{\prime 2} \mathrm{VC}(\mathcal{F})$,
	$$
	\mathbb{E}\left(\sup _{f \in \mathcal{F}} \sum_{k=1}^K f\left(\mathbf{Y}_k\right)-K \mathbb{E}\left(f\left(\mathbf{Y}_k\right)\right)\right) \leq \frac{K}{16} .
	$$
	Thus,  the second term in (\ref{l3}) is no more than $K/16+K/16=K/8$.  
	
	Back to the possibly contamination model, if $C \geq 16$, then $ K \geq 16|\mathcal{O}|$ and one can deduce that $$\sup _{f \in \mathcal{F}} \sum_{k=1}^K f\left(\mathbf{X}_k\right) \leq\sup _{f \in \mathcal{F}} \sum_{k=1}^K f\left(\mathbf{Y}_k\right)+ \frac{K }{16}.$$
	
	Putting everything together, we derive that, if $C \geq 256 C^{\prime 2}\vee 16$, the following event $\mathcal{E}$ has probability $\mathbb{P}(\mathcal{E}) \geq 1-\exp (-K / 128)$: for all $f \in \mathcal{F}$,
	$$
	\sum_{k=1}^K f\left(\mathbf{X}_k\right) \leq \frac{K }{16}+\frac{K }{8}+\frac{K }{16}=\frac{K }{4}.
	$$
	Hence, we have proved that, for any $x^* \in \mathcal{B}_0^*(\mathcal{U})$,
	$$
	\operatorname{Med}\left|x^*\left(\bar{X}_k-\mu\right)  \right| \leq r_K,
	$$
	whenever $ K\geq C(\mathrm{VC}(\mathcal{B}_0^*(\mathcal{U}))\vee|\mathcal{O}|).$ 	
	We conclude that taking $ \epsilon=r_{K} $, $\mathbb{S}(\epsilon)$ is nonempty as it contains $\mu$, at least on $\mathcal{E}$. By definition, $\widehat{\mu}_{K} \in \mathbb{S}(\epsilon)$ for this choice of $\epsilon$. Observe that for every $x^* \in \mathcal{B}_0^*(\mathcal{U})$ there is some index $j$ such that
	$$
	\left|x^*\left(\bar{X}_j\right)-x^*(\widehat{\mu}_{K})\right| \leq r_K \text { and }\left|x^*\left(\bar{X}_j\right)-x^*(\mu)\right| \leq r_K,
	$$
	because both conditions hold for more than half of the indices $j$. Thus,
	$$
	\left|x^*(\widehat{\mu}_{K})-x^*(\mu)\right| \leq\left|x^*\left(\bar{X}_j\right)-x^*(\widehat{\mu}_{K})\right|+\left|x^*\left(\bar{X}_j\right)-x^*(\mu)\right| \leq 2 r_K .
	$$
	
	Finally, recalling that $\|v\|=\sup _{x^* \in \mathcal{B}_0^*(\mathcal{U})} x^*(v)$, one has that
	$$
	\|\widehat{\mu}_{K}-\mu\|=\sup _{x^* \in \mathcal{B}_0^*(\mathcal{U})}\left|x^*(\widehat{\mu})-x^*(\mu)\right| \leq 2 r_K.
	$$
	Thus, for $ \widehat{\mu}_{K}\in \mathbb{S}(r_{K}) $, and $ K\geq C(\mathrm{VC}(\mathcal{B}_0^*(\mathcal{U}))\vee|\mathcal{O}|\vee 128\log (1 / \delta)), $ we obtain that with probability at least $ 1-\delta $,\[\|\widehat{\mu}_{K}-\mu\|\leq 8\sup _{v \in \mathcal{B}_0^*(\mathcal{U})} \mathbb{E}\left(\left\langle Y_1-\mu, v\right\rangle^2\right)^{1 / 2} \sqrt{\frac{K}{N}}.\]
\end{proof}

\subsection{Proof of Theorem \ref{c}}
The proof of Theorem \ref{c} is similar to that of Theorem \ref{m}, and we just give the first part in which we adjust the dual space and some coefficients.
\begin{proof}
	Let $$\mathcal{F}=\left\{\left(x_i\right)_{i \leq m} \rightarrow \mathbf{1}_{\left[\frac{1}{m} \sum_i x_i\otimes x_i -\Sigma, U\right] \geq r_K}, u \in \mathcal{B}_0^*\right\}, $$where
	$r_K=4 \sigma\sqrt{\frac{K}{N}}.$ The function $f \in \mathcal{F}$ are compositions of the function $x \rightarrow \frac{1}{m} \sum_i x_i\otimes x_i -\Sigma$ and the functions $x \rightarrow \mathbf{1}_{[ x, U] \geq r_K}$ for $U \in T$. The VC-dimension of the set of these compositions is smaller than the $\mathrm{VC}$-dimension of the set of indicator functions indexed by $T$. We just get $\mathrm{VC}(\mathcal{F})\leq\mathrm{VC}(T) \leq \mathrm{VC}(\mathcal{B}_0^*\otimes\mathcal{B}_0^*)$ , which is, by Lemma \ref{v}, bounded by $c_0 \mathrm{VC}(\mathcal{B}_0^*)$ for some constant $c_0$.
	
	By Markov's inequality, for any $u \in \mathcal{B}_0^*$, that is, $ U\in T $,
	$$
	\mathbb{P}\left(\left|\left[\frac{1}{m} \sum_i Y_i\otimes Y_i -\Sigma, U\right]\right| \geq r_K\right) \leq \frac{\mathbb{E}\left(\sum_{i \in B_1}\left[Y_i\otimes Y_i -\Sigma, U\right]^2\right)}{m^2 r_K^2} \leq \frac{1}{16} .
	$$	
	The rest of this proof is the same as that of the preceding theorem.
	There exists an absolute constant $C$ such that, if $K \geq C(d \vee|\mathcal{O}|)$, then, with probability larger than $1-\exp (-K / 128)$,
	$$
	\left\|\widehat{\Sigma}_\delta-\Sigma\right\| \leq 8 \sigma \sqrt{\frac{K}{N}} .
	$$
\end{proof}

\subsection{Proof of Theorem \ref{ttt}}

Before providing the proof of the theorem, for ease of manipulation, we first provide some lemmas without proof. The first lemma can be derived from the proof of Theorem \ref{m}, and the second is the Lemma 7.3 in \cite{chen2018robust}.

\begin{lemma}\label{l}
	For any $k \in [K]$, let $\mathbf{X}_k:=\left(X_i\right)_{i \in B_k}$ and $\mathbf{Y}_k:=\left(Y_i\right)_{i \in B_k}$. Let $\mathcal{F}$ be a Boolean class of functions satisfying the following two assumptions:
	\begin{itemize}
		\item For all $f \in \mathcal{F}, \mathbb{P}\left(f\left(\boldsymbol{Y}_1\right)=1\right) \leq \dfrac{1 }{4\alpha}$,
		\item $K \geq C(\mathrm{VC}(\mathcal{F}) \vee|\mathcal{O}|)$ where $C$ is a universal constant,
	\end{itemize}
	where $ \alpha>1 $ can be any constant. Then, with probability at least $1-\exp (-\frac{K}{8\alpha^{2}})$, for all $f \in \mathcal{F}$, there is at least $\dfrac{\alpha-1}{ \alpha} K $ blocks $B_k$ on which $f\left(\boldsymbol{X}_k\right)=0$.
	
\end{lemma}

Define a subset $H_{u, \eta}$ of $\mathbb{R}^d$ as $H_{u, \eta}=\left\{y:u^T y\leq u^T \eta\right\}$. We need the following concentration inequality for suprema of the empirical process indexed by these subsets $H_{u, \eta}$, where $u \in \mathcal{B}_0^*$ and $\eta \in \mathbb{R}^d$.

\begin{lemma}\label{7.3}
	For i.i.d. real-valued data $X_1, \ldots, X_n$ from distribution $\mathbb{P}$, and sufficiently large $n$, we have, with probability at least $1-\delta$,
	
	$$
	\sup _{u \in \mathcal{B}_0^*, \eta \in \mathbb{R}^d}\left|\mathbb{P}\left(H_{u, \eta}\right)-\mathbb{P}_n\left(H_{u, \eta}\right)\right| \leq \sqrt{\frac{1440 e \pi}{1-e^{-1}}} \sqrt{\frac{\mathrm{VC}\left(\mathcal{B}_0^*\right) }{n}}+\sqrt{\frac{\log (1 / \delta)}{2 n}},
	$$
	where $\mathbb{P}_n$ denotes the empirical distribution of $\left\{X_i\right\}_{i=1}^n$.
\end{lemma}

\begin{proof}[Proof of Theorem \ref{ttt}]
	Let $K \geq C(\mathrm{VC}(\mathcal{F}) \vee|\mathcal{O}|)$ with $C$ the universal constant from Lemma  \ref{l}, and let  $\mathcal{F}=\left\{\left(x_i\right)_{i \leq m} \rightarrow \mathbf{1}_{\left\langle\frac{1}{m} \sum_i x_i -\mu, v\right\rangle \geq r_{K,d}}, v \in \mathcal{B}_0^*\right\}$, $$
	r_{K,d}=8\sqrt{8} \sup _{v \in \mathcal{B}_0^*} \mathbb{E}\left(\left\langle Y_1-\mu, v\right\rangle^2\right)^{1 / 2} \sqrt{\frac{K}{N}}.
	$$ Based on the analysis in the previous Section \ref{3}, $\mathrm{VC}(\mathcal{F}) \leq c_0 \mathrm{VC}(\mathcal{B}_0^*)$ for some constant $c_0$.
	By Markov's inequality, for any $v \in \mathcal{B}_0^*$,
	$$
	\mathbb{P}\left(\left|\left\langle\frac{1}{m} \sum_{i \in B_1} Y_i-\mu, v\right\rangle\right| \geq r_{K,d}\right) \leq \frac{\mathbb{E}\left(\sum_{i \in B_1}\left\langle Y_i-\mu, v\right\rangle^2\right)}{m^2 r_{K,d}^2} \leq \frac{1}{32} .
	$$
	
	By Lemma \ref{l}, applied with $\mathcal{F}$ and $ \alpha=8 $, we derive that, with probability $\geq 1-\exp (-\frac{K}{2048})$, there is at least $\dfrac{7}{8} K $ blocks $B_k$ on which the following event happens:
	$$
	\sup _{v \in \mathcal{B}_0^*} \left|\left\langle\frac{1}{m} \sum_{i \in B_k} X_i-\mu, v\right\rangle\right| \leq r_{K,d}.
	$$
	We claim that, for any $ v \in \mathcal{B}_0^* $,
	$$
	\langle\mu-\hat{\mu}, v\rangle \leq  r_{K,d}.
	$$
	That is because, for any $a \in \mathbb{R}^d$ attains the deepest level, if there exists $v^* \in \mathcal{B}_0^*$ such that $\left\langle\mu-a, v^*\right\rangle> r_{K,d}$, then, on the above event,\begin{align*}
	\left\langle\bar{X}_{k}-a, v^*\right\rangle&=\left\langle\bar{X}_k-\mu, v^*\right\rangle+\left\langle\mu-a, v^*\right\rangle\\
	&>\left\langle\bar{X}_k-\mu, v^*\right\rangle+r_{K,d}\\
	&\geq 0
	\end{align*}
	holds for at least $\dfrac{7}{8} K $ blocks $B_k$, which means that, \[\sup_{v \in \mathcal{B}_0^*}\sum_{k=1}^{K}\mathbf{1}_{\left\langle \bar{X}_k-a,v\right\rangle > 0}\geq \sum_{k=1}^{K}\mathbf{1}_{\left\langle \bar{X}_k-a,v^*\right\rangle > 0}\geq  7K /8.\]
	
	We decompose the data $\left\{\bar{X}_k\right\}_{k=1}^K=\left\{\bar{X}_k\right\}_{k\in\mathcal{I}} \cup\left\{\bar{X}_k\right\}_{k\in\mathcal{J}}$, such that samples are clean for $ k\in\mathcal{I} $, and are corrupted for $ k\in\mathcal{J} $, then $ \dfrac{\left|\mathcal{I} \right| }{K}\geq 1-\varepsilon $. Using lemma \ref{7.3} for $ \left\{\bar{X}_k\right\}_{k\in\mathcal{I}} $ and $ \tilde{\mathbb{P}} $, where $ \tilde{\mathbb{P}}$ is the distribution of $ \bar{X}_k $, we have with probability at least $1-\delta$,
	
	$$
	\sup _\eta\left|\mathcal{D}\left(\eta, \tilde{\mathbb{P}}\right)-\mathcal{D}\left(\eta,\left\{\bar{X}_k\right\}_{k\in\mathcal{I}}\right)\right| \leq \sqrt{\frac{1440 e \pi}{1-e^{-1}}} \sqrt{\frac{d+1}{\left| \mathcal{I} \right| }}+\sqrt{\frac{\log (1 / \delta)}{2 \left| \mathcal{I} \right| }}.
	$$

	We lower bound $\mathcal{D}\left(\hat{\mu},\left\{\bar{X}_k\right\}_{k=1}^K\right)$ by
	$$
	\begin{aligned}
	\mathcal{D}\left(\hat{\mu},\left\{\bar{X}_k\right\}_{k=1}^K\right)
	& \geq  \mathcal{D}\left(\mu,\left\{\bar{X}_k\right\}_{k=1}^K\right)\\
	& \geq \frac{\left| \mathcal{I} \right| }{K}\mathcal{D}\left(\mu,\left\{\bar{X}_k\right\}_{k\in\mathcal{I}}\right)\\
	& \geq \frac{\left| \mathcal{I} \right|}{K}\left( \mathcal{D}\left(\mu, \tilde{\mathbb{P}}\right)- \sqrt{\frac{1440 e \pi}{1-e^{-1}}} \sqrt{\frac{d+1}{\left| \mathcal{I} \right| }}-\sqrt{\frac{ \log (1 / \delta)}{2\left| \mathcal{I} \right| }} \right) .
	\end{aligned}
	$$
	The second inequality is due to the property of depth function that
	
	$$ K\mathcal{D}\left(\eta,\left\{\bar{X}_k\right\}_{k=1}^K\right) \geq \left| \mathcal{I} \right| \mathcal{D}\left(\eta,\left\{\bar{X}_k\right\}_{k\in\mathcal{I}}\right),
	$$
	for any $\eta \in \mathbb{R}^d$. Since $ \tilde{\mathbb{P}}$ represents the distribution of $m $ independent and identically distributed random vectors drawn from the distribution $\mathbb{P}$, the Central Limit Theory implies that $\tilde{\mathbb{P}}$  approximates a normal distribution, which is symmetric. Therefore, when $ m $ is sufficiently large, and $ K\geq C(\mathrm{VC}(\mathcal{F}) \vee|\mathcal{O}|) $, we have
	$$\mathcal{D}\left(\hat{\mu},\left\{\bar{X}_k\right\}_{k=1}^K\right)\geq\dfrac{(1-\varepsilon)}{2}\mathcal{D}\left(\mu, \tilde{\mathbb{P}}\right)\geq 1/8,$$
	for sufficiently large $ m $ and $ K $, which means that $$ \sup_{v \in \mathcal{B}_0^*}\sum_{k=1}^{K}\mathbf{1}_{\left\langle \bar{X}_k-\hat{\mu},v\right\rangle > 0}< 7K/8  .$$
	
	Therefore, $a \neq \hat{\mu}$. Hence, we have proved the claim, $\langle\mu-\hat{\mu}, v\rangle \leq  r_{K,d}. $ Take the supremum over \(v \in \mathcal{B}_0^*\) and since $\sup _{v \in \mathcal{B}_0^*}\left\langle\mu-\hat{\mu}, v\right\rangle=\|\mu-\hat{\mu}\|$, we conclude the proof.

\end{proof}

\bibliography{momref}

\end{document}